\documentclass[oneside,12pt]{amsart}
\usepackage[T2A]{fontenc}
\usepackage[cp1251]{inputenc}
\usepackage{amssymb}
\usepackage{amsxtra}
\usepackage{amsmath}
\usepackage{amstext,amsfonts}

\usepackage{amsthm}
\usepackage[all]{xy}

\usepackage{hyperref}
\usepackage{tikz}

\usepackage[russian,english]{babel}

\DeclareMathOperator{\GL}{GL}
\DeclareMathOperator{\Sp}{Sp}
\DeclareMathOperator{\SL}{SL}
\DeclareMathOperator{\EE}{E}

\DeclareMathOperator{\Lie}{Lie}
\DeclareMathOperator{\Cent}{Cent}
\DeclareMathOperator{\rad}{rad}
\DeclareMathOperator{\Norm}{Norm}
\DeclareMathOperator{\Hom}{Hom}
\DeclareMathOperator{\Spec}{Spec}

\DeclareMathOperator{\Gm}{{\mathbf G}_m}
\newcommand{\id}{\text{\rm id}}

\DeclareMathOperator{\coker}{coker\,}
\DeclareMathOperator{\im}{im\,}
\DeclareMathOperator{\Aut}{Aut\,}

\DeclareMathOperator{\Gal}{Gal}

\newcommand{\Aff}{\mathbb {A}}

\newcommand{\ad}{{ad}}
\newcommand{\scl}{{sc}}
\newcommand{\ha}{{\widetilde{\alpha}}}

\DeclareMathOperator{\ZZ}{{\mathbb Z}}

\DeclareMathOperator{\NN}{{\mathbb N}}
\DeclareMathOperator{\PP}{{\mathbb P}}

\DeclareMathOperator{\fppf}{\text{\it fppf}}
\DeclareMathOperator{\FF}{{\mathbb F}}

\newtheorem{lem}{Lemma}[section]
\newtheorem*{lem*}{Lemma}
\newtheorem*{thm*}{Theorem}
\newtheorem{thm}[lem]{Theorem}
\newtheorem{cor}[lem]{Corollary}
\newtheorem{defn}{Definition}

\let\l\left
\let\r\right

\newcommand{\eps}{\varepsilon}
\newcommand{\st}{\scriptstyle}

\newcommand{\ee}{\hbox{${\bf (E)}$}}

\makeatletter

\@addtoreset{equation}{section}
\@addtoreset{defn}{section}

\makeatother

\begin{document}

\title{Homotopy invariance of non-stable $K_1$-functors}
\author{A. Stavrova}\thanks{The author was supported by the J.E. Marsden postdoctoral fellowship of the Fields Institute,
by the RFBR grants 12-01-33057, 12-01-31100,
10-01-90016, 10-01-00551, and by the research program 6.38.74.2011 ``Structure theory and geometry
of algebraic groups and their applications in representation theory and algebraic K-theory'' at St.
Petersburg State University.}
\address{Fields Institute for Research in Mathematical Sciences\\
Toronto, Canada\\ and Department of Mathematics and Mechanics\\
St. Petersburg State University\\Russia}
\email{anastasia.stavrova@gmail.com}

\selectlanguage{english}
\maketitle

\begin{abstract}
Let $G$ be a reductive algebraic group over a field $k$, such that every semisimple normal subgroup of $G$
has isotropic rank $\ge 2$, i.e. contains $(\Gm)^2$. Let $K_1^G$ be the non-stable $K_1$-functor associated to $G$,
also called the Whitehead group of $G$. We show that $K_1^G(k)=K_1^G(k[X_1,\ldots,X_n])$ for any $n\ge 1$.
If $k$ is perfect, this implies that $K_1^G(R)=K_1^G(R[X])$ for any regular $k$-algebra $R$.
If $k$ is infinite perfect, one also deduces that $K_1^G(R)\to K_1^G(K)$ is injective for
any local regular $k$-algebra $R$ with the fraction field $K$.
\end{abstract}

\section{Introduction}
Let $A$ be a unital commutative ring. The elementary subgroup $E_l(A)$ of $\GL_l(A)$ is the subgroup generated by the
elementary transvections $e+te_{ij}$, $1\le i\neq j\le l$, $t\in A$.
For any reductive group scheme $G$ over $A$, satisfying a suitable isotropy condition,
one defines an analogous elementary subgroup $E(A)$ of the group of $A$-points $G(A)$, as the subgroup
generated by the $A$-points of unipotent radicals of parabolic subgroups of $G$; see \S~\ref{ssec:eldef} or~\cite{PS}
for a formal definition. In particular, if $A=k$ is a field, $E(k)$ is nothing but the group $G(k)^+$ introduced
by J. Tits~\cite{Tits64}. If $G$ is a Chevalley group with root system $\Phi$, then $E(A)=E(\Phi,A)$ is the group
generated by the elementary root subgroups $x_{\alpha}(A)$, $\alpha\in\Phi$.

The study of the quotient $K_1^G(A)=G(A)/E(A)$ is one of the central problems in the theory of reductive
algebraic groups; in the field case it is called the Kneser--Tits problem, see~\cite{Gil} and references therein.
The functor $K_1^G(-)=G(-)/E(-)$ on the category of commutative $A$-algebras
is called the non-stable, or unstable, $K_1$-functor associated to $G$, or
the Whitehead group of $G$. The name is
due to the fact that the functors $K_1^G$ are similar in many aspects to the algebraic $K_1$-functor
$K_1(A)=\lim\limits_{l}\GL_l(A)/E_l(A)$ of H. Bass.
The study of $K_1^{G}$ goes back to Bass' founding paper~\cite{Bass}, where $K_1^{\GL_l}$ was considered
in relation to stabilization problems.
The principal bibliography on the subsequent study of $K_1^G$ for different reductive groups $G$
can be found in~\cite{Abe,Gil,GMV-Sp,BHV,PS,W10}.

In particular, recently F. Morel~\cite{Mo} and
M. Wendt~\cite{W10} showed that if $G$ is a Chevalley group of rank $\ge 2$, then $K_1^G$
is the $0$-th $\Aff^1$-homotopy group of $G$ in the sense of $\Aff^1$-homotopy theory of Morel--Voevodsky.
The proof relies on the $\Aff^1$-invarince of $K_1^G$ on the category of regular algebras
over a field, that is, on the isomorphism $K_1^G(A)\cong K_1^G(A[X])$, proved by A. Suslin and E. Abe in~\cite{Sus,Abe}.
The present paper is devoted to generalization of the work of Suslin and Abe to isotropic reductive groups.
Its results are expected to allow extension of~\cite{W10} to this context, see~\cite{VW}.

Our first result is the following theorem.
It means, essentially, that $K_1^G$ satisfies the gluing property for the standard covering of $\PP^1$
by two copies of $\Aff^1$ over an affine base. Note that $K_1^G(\PP^1_A)\cong K_1^G(A)$ for any commutative ring $A$.

\begin{thm}\label{th:P^1}
Let $A$ be a commutative ring, and let $G$ be a reductive group scheme over $A$, such that every semisimple normal subgroup
of $G$ is isotropic. Assume moreover that for any maximal ideal $m\subseteq A$,
 every semisimple normal subgroup of $G_{A_m}$ has isotropic rank $\ge 2$, i.e. contains $(\Gm)^2$.
Then the sequence of pointed sets
$$
1\longrightarrow K_1^G(A)\xrightarrow{\st g\mapsto (g,g)} K_1^G(A[X])\times K_1^G(A[X^{-1}])
\xrightarrow{\st (g_1,g_2)\mapsto g_1{g_2}^{-1}} K_1^G(A[X,X^{-1}])
$$
is exact.
\end{thm}


This theorem was first proved by
A. Suslin~\cite[Theorem 5.1]{Sus} for $G=\GL_l$  ($l\ge 3$), and, using similar methods, by A. Suslin and V. Kopeiko
for $O_{2l}$ ($l\ge 3$)~\cite[Theorem 6.8]{Sus-K-O1} and $\Sp_{2l}$ ($l\ge 2$)~\cite[Theorem 3.9]{K-stab}. Later
E. Abe~\cite[Theorem 2.16]{Abe} generalized it to almost all Chevalley groups. More
precisely, Abe proved the same statement for all split simply connected absolutely almost simple groups of rank $\ge 2$
and any commutative ring $A$, excluding
the groups of type $C_l$ over $A$ with $2\not\in A^\times$, and all groups of type $B_l$ and $G_2$. Some steps
of our proof are inspired by Abe's work, but we do not rely on his results. In particular, we remove the above restrictions
on type in the split case.

The next theorem is our main result.
On one hand, it generalizes the equality
$G(k[X_1,\ldots,X_n])=E(k[X_1,\ldots,X_n])$ that was proved by
E. Abe for all simply connected Chevalley groups with the same exceptions as above~\cite[Theorem 3.5]{Abe}.
For $G=\GL_l,\SL_l,\Sp_{2l}$ somewhat stronger versions of this result were previously obtained in the above-mentioned
works of Suslin and Kopeiko, and in~\cite{GMV-Sp,K-SL-Laur,K-Sp-Laur,K-lett}. On the other hand,
it generalizes a theorem of C. Soul\'e~\cite{Sou} and B. Margaux~\cite{M}
stating that $G(k[X])=G(k)E(k[X])$ for any isotropic simply connected absolutely almost simple group $G$ over $k$.

\begin{thm}\label{th:whitehead-field}
Let $G$ be a reductive
group scheme over a field $k$, such that every semisimple normal subgroup of $G$ contains $(\Gm)^2$.
Then
$$G(k[X_1,\ldots,X_n])=G(k)E(k[X_1,\ldots,X_n])\qquad\mbox{for any}\quad n\ge 1.
$$ That is,
$K_1^G(k)\cong K_1^G(k[X_1,\ldots,X_n])$.
\end{thm}


The proof of Theorem~\ref{th:whitehead-field} goes by induction on $n$. The case $n=1$ is covered
by Margaux-- Soul\'e theorem (see also Theorem~\ref{th:MS-red} below).
The inductive step uses Theorem~\ref{th:P^1}, and the isomorphism $K_1^G(k)\xrightarrow{\cong} K_1^G(k(X))$ of~\cite[Th. 5.8]{Gil},
which holds if $G$ is a simply connected semisimple
group. For simply connected groups we actually prove a stronger isomorphism
$$
K_1^G(k)\cong K_1^G\bigl(k[X_1,\ldots,X_n,Y_1,Y_1^{-1},\ldots,Y_m,Y_m^{-1}]\bigr)\qquad\mbox{for any}\quad m,n\ge 1;
$$
see Corollary~\ref{cor:Laurent}.

If the base field $k$ is perfect, Theorem~\ref{th:whitehead-field} is part of the following more general result.

\begin{thm}\label{th:whitehead-geometric}
Let $k$ be a perfect field, and let $G$ be a reductive group scheme over $k$,
such that every semisimple normal subgroup of $G$ contains $(\Gm)^2$.
Let $A$ be a regular ring containing $k$. Then there is a natural isomorphism
$K_1^G(A)\cong K_1^G(A[X])$.
\end{thm}


To prove this theorem, we use Popescu's theorem to reduce to the case of
a regular $k$-algebra essentially of finite type. In this generality, the result follows from the previous statements
and Lindel's lemma on \'etale neighbourhoods~\cite{L}. This approach is due to T. Vorst~\cite{Vo},
who considered the case of $\GL_l$; the same argument was used by Abe~\cite{Abe} for Chevalley groups
under the same restrictions on type as mentioned above, and for $\Sp_{2l}$ by V. Kopeiko in~\cite{K-reg,K-lett}.
Note that in these three cases the ground field $k$ was not supposed to be perfect, since for split groups
the case of a regular algebra essentially of finite type over a non-perfect field reduces to the case of that
over $\FF_p$~\cite[Proof of Th. 3.3]{Vo}.
M. Wendt~\cite[Prop. 4.8]{W10} suggested a way to extend Abe's result
to Chevalley groups of types $B_l$, $C_l$ and $G_2$, using stabilization results of E. Plotkin (actually, even in case of an excellent Dedekind ring $k$),
but his proof is known to be incomplete~\cite{Step}.

Note that the isomorphism $K_1^G(R)\xrightarrow{\cong}K_1^G(R[X])$ implies that $K_1^G(R)$ coincides also with the first
Karoubi---Villamayor $K$-group of $R$ with respect to $G$, as defined by J.F. Jardine~\cite{J} following S.M. Gersten~\cite{Ge}; see~\cite{W10}
or Lemma~\ref{lem:KV}.

Our last result is deduced from Theorem~\ref{th:whitehead-field} by means of a general theorem of J.-L. Colliot-Th\'el\`ene and M.
Ojanguren~\cite[Th\'eor\`eme 1.1]{CTO}.

\begin{thm}\label{th:inj}
Let $k$ be an infinite field, let $G$ be a reductive group scheme over $k$,
such that every semisimple normal subgroup of $G$ contains $(\Gm)^2$.  Let $A$ be a local regular ring containing $k$,
and let $K$ be the field of fractions of $A$.
Assume that $k$ is perfect or that $A$ is a local ring of a smooth algebraic variety over $k$.
Then the natural homomorphism $K_1^G(A)\to K_1^G(K)$ is injective.
\end{thm}

It should be possible to extend all results of the present paper to isotropic simply connected absolutely almost simple groups $G$
which are defined over a semilocal regular ring $R$ containing a (perfect, infinite) field $k$, and not
over $k$ itself, by means of the techniques employed in~\cite{PaSV}. We plan to address this question in the near future.

\medskip

The author thanks Ch. Weibel and Rutgers University for the hospitality during her visits in 2011 and 2012,
when a part of this paper was written. She is also grateful to M. Wendt for pointing out some misprints in a previous version
of this text.

\section{Elementary subgroup of an isotropic reductive group}

Let $G$ be a reductive algebraic group over a commutative ring $A$.

\subsection{Definition of the elementary subgroup and Suslin's local-global principle}\label{ssec:eldef}

Let $P$ be a parabolic subgroup of $G$.
Since the base $\Spec A$ is affine, the group $P$ has a Levi subgroup $L_P$~\cite[Exp.~XXVI Cor.~2.3]{SGA}.
There is a unique parabolic subgroup $P^-$ in $G$ which is opposite to $P$ with respect to $L_P$,
that is $P^-\cap P=L_P$, cf.~\cite[Exp. XXVI Th. 4.3.2]{SGA}.  We denote by $U_P$ and $U_{P^-}$ the unipotent
radicals of $P$ and $P^-$ respectively. Note that $L_P$ normalizes  $U_P$ and $U_{P^-}$.

\begin{defn}
The \emph{elementary subgroup} $E_P(A)$ corresponding to $P$ is the subgroup of $G(A)$
generated as an abstract group by $U_P(A)$ and $U_{P^-}(A)$.
\end{defn}

Note that if $L'_P$ is another Levi subgroup of $P$,
then $L'_P$ and $L_P$ are conjugate by some element $u\in U_P(A)$~\cite[Exp. XXVI Cor. 1.8]{SGA}, hence
$\EE_P(A)$ does not depend on the choice of a Levi subgroup or of an opposite subgroup
$P^-$, respectively. We suppress the particular choice of $L_P$ in this context, and sometimes even write
$U_P^-$ instead of $U_{P^-}$.

\begin{defn}
A parabolic subgroup $P$ in $G$ is called
\emph{strictly proper}, if it intersects properly every normal semisimple subgroup of $G$.
\end{defn}

Equivalently, $P$ is strictly proper, if for every maximal ideal $m$ in $A$ the image of $P_{A_m}$ in
$G_i$ under the projection map is a proper subgroup in $G_i$, where $G^{ad}_{A_m}=\prod_i G_i$
is the decomposition of the adjoint semisimple group $G_{A_m}^{ad}$ into a product of absolutely almost simple groups.

\begin{defn}
We say that $G$ over $A$ satisfies the assumption \ee, if $G$ contains
a strictly proper parabolic $A$-subgroup $P$, and
 for any maximal ideal $m$ in $A$, the group $G_{A_m}$ contains at least two different strictly proper
parabolic $A_m$-subgroups $P_1\le P_2$.
\end{defn}

Note that the second part of the above definition can be restated as follows:
every normal semisimple subgroup of $G_{A_m}$ contains $(\Gm)^2$, or, equivalently,
all irreducible components of the
root system of $G_{A_m}$ relative to a maximal split subtorus in the sense of~\cite[Exp. XXVI, \S 7]{SGA} are of rank $\ge 2$.

The main result of~\cite{PS} says that if $G$ satisfies \ee,
then $E(A)=E_P(A)$ is independent of the choice of a strictly proper parabolic subgroup $P$, and
in particular, is normal in $G$.

\begin{thm}\label{th:PS-normality}\cite[Lemma 12, Theorem 1; SGA3]{PS},
Let $G$ be a reductive algebraic group over a commutative ring $A$, and let $R$ be an arbitrary commutative $A$-algebra.

(i) If $A$ is a semilocal ring, then the subgroup $\EE_P(R)$ of $G(R)$ is the same for any
minimal parabolic $A$-subgroup $P$ of $G$. If, moreover, $G$ contains a strictly proper parabolic $A$-subgroup,
the subgroup $\EE_P(R)$ is the same for any strictly proper parabolic $A$-subgroup $P$.

(ii) If $G$ satisfies \ee~ over $A$, then the subgroup $\EE_P(R)$ of $G(R)=G_R(R)$ is the same for any
strictly proper parabolic $R$-subgroup $P$ of $G_R$.

In all these cases $\EE(A)=\EE_P(A)$ is normal in $G(A)$.
\end{thm}
\begin{proof}
To show (i), recall that by~\cite[Cor. 5.7]{SGA} every parabolic $A$-subgroup $P$ of $G$ contains
a minimal prabolic $A$-subgroup $P_{min}$. If $P$ is strictly proper, then $P_{min}$ is
strictly proper as well. By~\cite[Lemma 12]{PS}, over any commutative ring,
if one strictly proper parabolic subgroup contains another, then the respective elementary subgroups coincide;
hence $E_P(R)=E_{P_{min}}(R)$. By~\cite[Cor. 5.2 and Cor. 5.7]{SGA}, any other minimal
parabolic $A$-subgroup $Q_{min}$ of $G$ is conjugate to $P_{min}$ by an element of $E_{P_{min}}(A)$.
Consequntly, $E_{P_{min}}(R)$ is independent of the choice of a minimal parabolic $A$-subgroup $P_{min}$
of $G$.

The statement (ii) is precisely~\cite[Theorem 1]{PS}, taking into account that if $G$ satisfies \ee~over $A$,
it also satisfies \ee~over $R$.

Since any conjugate of a minimal or strictly proper parabolic subgroup
is, respectively, minimal or strictly proper, both (i) and (ii) imply that $\EE_P(A)$ is normal in $G(A)$.
\end{proof}

The facts from~\cite{SGA}, used in the first part of Theorem~\ref{th:PS-normality}, are corollaries of the
so-called Gauss decomposition in isotropic reductive groups over semi-local rings~\cite[Exp. XXVI, Th\'eor\`eme 5.1]{SGA}.
Namely, for any reductive group $G$ over a semi-local ring $A$, and any pair of opposite parabolic subgroups $P^\pm$
of $G$ with a common Levi subgroup $L_P$, one has
$$
G(A)=U_{P^+}(A)U_{P^-}(A)L_P(A)U_{P^+}(A).
$$

Given this result, the proof of~\cite[Theorem 1]{PS} consists in reducing the case of
a general commutative ring $A$ to the case of a local ring. This is done by showing that
$G$ satisfies what we call Suslin's local-global principle. It was proved for $G=\GL_l$, $n\ge 3$, in~\cite[Th. 3.1]{Sus}.
We denote by $F_m:A\to A_m$ the localization maps.

\begin{defn}{\bf (Suslin's local-global principle)}
Let $G$ be a reductive group scheme over a commutative ring $A$, and let $E(A)$ be the elementary subgroup of
$G(A)$. We say that $G$ satisfies \emph{Suslin's local-global principle}, if for any $g(X)\in G(A[X])$ such that $g(0)\in E(A)$ and
$F_m(g(X))\in E(A_m[X])$ for all maximal ideals $m$ of $A$, one has $g(X)\in E(A[X])$.
\end{defn}

Note that Suslin based his proof of the above statement for $\GL_l$ on the ideas of Quillen from~\cite{Q} (e.g.~\cite[Lemma 1]{Q}).
For the case of split semisimple (=Chevalley) groups the same result was obtained by Abe in~\cite[Th. 1.15]{Abe}.
R. Basu has treated certain isotropic reductive groups of classical type under the assumption
that they are locally split (\cite{Basu}; see also~\cite{BBR}).

The most general known result for reductive groups is as follows:

\begin{lem}\label{lem:PS-17}\cite[Lemma 17]{PS}
Let $G$ be a reductive group scheme over a commutative ring $A$, satisfying the condition \ee. Then Suslin's local-global
principle holds for $G$.
\end{lem}


\subsection{Relative roots and relative root subschemes}\label{ssec:rel}
The results of~\cite{PS} were obtained using the calculus of relative roots and relative roots subschemes
of isotropic reductive groups. Here we recall their main properties proved in~\cite{PS,LS}.
They will be freely used in the rest of the paper.


Let $A$ be a commutative ring. Let $G$ be an isotropic reductive group scheme over $A$. Let
$P=P^+$ and $P^-$ be two opposite strictly proper
parabolic $A$-subgroups of $G$, with the common Levi subgroup $L_P=P\cap P^-$.
Relative root subschemes of $G$ with respect to $P=P^+$, actually,
depend on the choice of $P^-$ or $L_P$, but their essential properties stay the same,
so we will usually omit $P^-$ from the notation.

It was shown in~\cite{PS} that we can represent $\Spec A$ as a finite disjoint union
$$
\Spec A=\coprod\limits_{i=1}^m\Spec A_i,
$$ so that
the following conditions hold for $i=1,\ldots, m$:\\
\indent$\bullet$  the root system of $G_{\overline{k(s)}}$ is the same for all $s\in\Spec A_i$;\\
\indent$\bullet$  the type of the parabolic subgroup $P_{\overline{k(s)}}$ of
$G_{\overline{k(s)}}$ is the same for all $s\in\Spec A_i$;\\
\indent$\bullet$ if $S_i/A_i$ is a Galois extension of rings such that $G_{S_i}$ is of inner type,
then for any $s\in\Spec A_i$ the Galois group $\Gal(S_i/A_i)$ acts on the Dynkin diagram $D_i$ of
$G_{\overline{k(s)}}$ via the same subgroup of $\Aut(D_i)$.

The relative roots and relative root subschemes of $G$ are correctly defined over each $A_i$, $1\le i\le m$, that is, only locally
in the Zariski topology on $\Spec A$. However, since
$$
E_P(A)=\l<U_P(A),U_{P^-}(A)\r>=\prod\limits_{i=1}^m E_P(A_i),
$$
all properties of the elementary subgroup proved using relative root subschemes are usually easy to extend to
the general case. From here until the end of this section, assume that $A=A_i$ for some $i$.

Denote by $\Phi$ the root system of $G$, by $\Pi$ a set of simple roots of $\Phi$, by $D$
the corresponding Dynkin diagram. Then the $*$-action on $D$ is determined by a subgroup $\Gamma$
of $\Aut D$. Let $J$ be the subset of $\Pi$ such that $\Pi\setminus J$ is the type
of $P_{\overline{k(s)}}$ (that is, the set of simple roots of the Levi sugroup $L_{\overline{k(s)}}$).
Then $J$ is $\Gamma$-invariant.
Consider the projection
$$
\pi=\pi_{J,\Gamma}\colon\ZZ \Phi\longrightarrow \ZZ\Phi/\l<\Pi\setminus J;\ \alpha-\sigma(\alpha),\ \alpha\in J,\ \sigma\in\Gamma\r>.
$$
\begin{defn}
The set $\Phi_P=\pi(\Phi)\setminus\{0\}$ is called the system of \emph{relative roots} with respect to the parabolic
subgroup $P$. The \emph{rank} of $\Phi_P$ is the rank of $\pi(\ZZ \Phi)$ as a free abelian group.
\end{defn}

If $A$ is a local ring and $P$ is a minimal parabolic subgroup of $G$,
then $\Phi_P$ can be identified with the relative root system of $G$ in the sense of~\cite[Exp. XXVI \S 7]{SGA}
or~\cite{BoTi} in the field case, see also~\cite{PS,S-thes}.

In~\cite{PS}, we associated to any relative root $\alpha\in\Phi_P$
a finitely generated projective $A$-module $V_\alpha$ and a closed embedding
$$
X_\alpha:W(V_\alpha)\to G,
$$
where $W(V_\alpha)$ is the affine group scheme over $A$ defined by $V_\alpha$,
which is called a {\it relative root subscheme} of $G$.
These subschemes possess several nice properties
similar to that of elementary root subgroups of a split group, which we summarize
below.

\begin{thm*}\cite[Theorem 2, Lemma 9]{PS}
Let $\alpha\in\Phi_P$.

(i) There exist degree $i$ homogeneous polynomial maps $q^i_\alpha:V_\alpha\oplus V_\alpha\to V_{i\alpha}$, i>1,
such that for any $A$-algebra $A'$ and for any
$v,w\in V_\alpha\otimes_A A'$ one has
\begin{equation}\label{eq:sum}
X_\alpha(v)X_\alpha(w)=X_\alpha(v+w)\prod_{i>1}X_{i\alpha}\l(q^i_\alpha(v,w)\r).
\end{equation}

(ii) For any $g\in L_P(A)$, there exist degree $i$ homogeneous polynomial maps
$\varphi^i_{g,\alpha}\colon V_\alpha\to V_{i\alpha}$, $i\ge 1$, such that for any $A$-algebra $A'$ and for any
$v\in V_\alpha\otimes_A A'$ one has
$$
gX_\alpha(v)g^{-1}=\prod_{i\ge 1}X_{i\alpha}\l(\varphi^i_{g,\alpha}(v)\r).
$$
If $g$ is contained
in the central subtorus $\rad(L)(A)$, then $\varphi^1_{g,\alpha}$ is multiplication by a scalar, and
all $\varphi^i_{g,\alpha}$, $i>1$, are trivial.

(iii) \emph{(generalized Chevalley commutator formula)} For any $\alpha,\beta\in\Phi_P$ such that $m\alpha\neq -k\beta$ for
any $m,k\ge 1$,
there exist polynomial maps
$$
N_{\alpha\beta ij}\colon V_\alpha\times V_\beta\to V_{i\alpha+j\beta},\ i,j>0,
$$
homogeneous of degree $i$ in the first variable and of degree $j$ in the second
variable, such that for any $A$-algebra $A'$ and for any
for any $u\in V_\alpha\otimes_A A'$, $v\in V_\beta\otimes_A A'$ one has
\begin{equation}\label{eq:Chev}
[X_\alpha(u),X_\beta(v)]=\prod_{i,j>0}X_{i\alpha+j\beta}\bigl(N_{\alpha\beta ij}(u,v)\bigr)
\end{equation}
\end{thm*}





Apart from the above properties of relative root subschemes, we will use the following Lemma, which appeared
first as~\cite[Lemma 10]{PS}, and was strengthened in~\cite{LS}.

\begin{lem*}\cite[Lemma 2]{LS}
Consider $\alpha,\beta\in\Phi_P$ satisfying $\alpha+\beta\in\Phi_P$ and $m\alpha\neq -k\beta$ for any $m,k\ge 1$.
Denote by $\Phi^0$ an irreducible component of $\Phi$ such that $\alpha,\beta\in\pi(\Phi^0)$.

{\rm (1)} In each of the following cases

\quad {\rm (a)} structure constants of $\Phi^0$ are invertible in the base ring $A$ (for example,
if $\Phi^0$ is simply laced);

\quad {\rm (b)} $\alpha\neq \beta$ and $\alpha-\beta\not\in\Phi_P$;

\quad {\rm (c)} $\Phi^0$ is of type $B_l$, $C_l$, or $F_4$,
and $\pi^{-1}(\alpha+\beta)$ consists of short roots;

\quad {\rm (d)} $\Phi^0$ is of type $B_l$, $C_l$, or $F_4$, and there exist long roots $\alpha\in\pi^{-1}(\alpha)$, $\beta\in\pi^{-1}(\beta)$
such that $\alpha+\beta$ is a root;

the map $\ N_{\alpha\beta 11}: V_\alpha\times V_\beta\to V_{\alpha+\beta}\ $
is surjective.

{\rm (2)} If $\alpha-\beta\in\Phi_P$ and $\Phi^0\neq G_2$, then
$$
\im N_{\alpha\beta 11}+\im N_{\alpha-\beta,2\beta,1,1}+\sum_{v\in V_\beta}\im (N_{\alpha-\beta,\beta,1,2}(-,v))=V_{\alpha+\beta},
$$
where $\im N_{\alpha-\beta,2\beta,1,1}=0$ if $2\beta\not\in\Phi_P$.
\end{lem*}

Assume that $G$ satisfies \ee over $A$, and $P$ is strictly proper.
In a strict analogy with the split case, for any $A$-algebra $A'$ we have
$$
E(A')=E_P(A')=\l<U_P(A),U_{P^-}(A)\r>=\l<X_\alpha(V_\alpha\otimes_A A'),\ \alpha\in\Phi_{P}\r>.
$$
(see~\cite[Lemma 6]{PS}).

For any $\alpha\in\Phi_P$, we denote by $U_{(\alpha)}$ the closed subscheme $\prod\limits_{k\ge 1}X_{k\alpha}$ of $G$
so that we have
$$
U_{(\alpha)}(A')=\l<X_{k\alpha}(V_{k\alpha}\otimes_A A'),\ k\ge 1\r>
$$
for any $A'/A$. Here $X_{k\alpha}$ is assumed to be trivial if $k\alpha\not\in\Phi_P$. This notation
coincides with that of~\cite{BoTi} in case of isotropic reductive groups over a field.

\begin{defn}\label{def:E(A,I)}
Let $I$ be any ideal of the ring $A$.
We denote
$$
\begin{array}{l}
G(A,I)=\ker \bigl(G(A)\to G(A/I)\bigr),\\
E_P^*(A,I)=G(A,I)\cap E_P(A),\\
E_P(I)=\l<X_\alpha(IV_\alpha),\ \alpha\in\Phi_P\r>=\l<X_\alpha(V_\alpha\otimes_A I),\ \alpha\in\Phi_P\r>,\\
E_P(A,I)=E_P(I)^{E_P(A)}=\mbox{ the normal closure of $E_P(I)$ in }E_P(A).
\end{array}
$$
If \ee~is satisfied, we denote the same groups by $E^*(A,I)$, $E(I)$, and $E(A,I)$ respectively.
\end{defn}

For any $\alpha\in\Phi_P$ there exists by~\cite[Exp. XXVI Prop. 6.1]{SGA} a closed connected smooth subgroup $G_\alpha$ of $G$
such that for any $s\in\Spec A$, $(G_\alpha)_{\overline{k(s)}}$ is the standard reductive subgroup of $G_{\overline{k(s)}}$
corresponding to the  root subsystem $\pi^{-1}(\{\pm\alpha\}\cup\{0\})\cap\Phi$. The group $G_\alpha$ is an isotropic reductive
group ``of isotropic rank 1'', having two opposite parabolic subgroups $L_P\cdot U_{(\alpha)}$ and $L_P\cdot U_{(-\alpha)}$.

We denote by $E_\alpha(A)$ the subgroup of $G(A)$
generated by $U_{(\alpha)}(A)$ and $U_{(-\alpha)}(A)$. Note that we don't know if $E_\alpha(A)$ is normal
in $G_\alpha(A)$, and, generally speaking, it depends on the choice of the initial parabolic subgroup $P$ of $G$.

\subsection{Factorization lemma for the elementary subgroup}\label{ssec:E_fg}

Suslin's local-global principle is closely related to the following factorization lemma, first
proved in~\cite[Lemma 3.7]{Sus} for $\GL_l$, and in~\cite[Lemma 3.2]{Abe} for split groups.
It was originally inspired by yet another step in the proof of
Quillen's local-global principle for projective modules~\cite[Theorem 1]{Q}.

\begin{lem}\label{lem:A_fg}
Let $G$ be an isotropic reductive group over a commutative ring $A$, and let $P$ be a strictly proper parabolic subgroup
of $G$. Assume that the relative root
subschemes with respect to $P$ are correctly defined over $A$, as in subsection~\ref{ssec:rel} above,
and that all irreducible components of $\Phi_P$ are of rank at least $2$; in particular, $G$ satisfies \ee. 

Let $f,g\in A$ be such that $fA+gA=A$. If $x\in E(A_{fg})$, then there exist
$x_1\in E(A_f)$ and $x_2\in E(A_g)$ such that $x=F_g(x_1)F_f(x_2)$.
\end{lem}


To prove this lemma, we need the following extensions of~\cite[Lemmas 15 and 16]{PS}.

\begin{lem}\label{lem:PS-15'}
Fix $s\in A$, and let $F_s:G(A[Z])\to G(A_s[Z])$ be the localization homomorphism. For any $g(Z)\in E(A_s[Z],ZA_s[Z])$ there
exist such $h(Z)\in E(A[Z],ZA[Z])$ and $k\ge 0$ that $F_s(h(Z))=g(s^kZ)$.
\end{lem}
\begin{proof}
Let $S=\{s^k,\ k\ge 0\}$ be the set of all powers of $s$ in $A$. One can prove exactly as in~\cite[Lemma 15]{PS}, that
for any $g(Z)\in E(A_S[Z],ZA_S[Z])$ there exist such
$f(Z)\in E(A[Z],ZA[Z])$ and $s^k\in S$ that $F_h(f(Z))=g(s^kZ)$. Indeed, in that Lemma, the localization was taken with respect
to the subset $S$ of the base ring $A$ which was a complement of a maximal ideal, and not a set of powers of one element.
However, the only use of the fact that $A_S$ was a local ring was that $G_{A_S}$ contained a parabolic subgroup
whose relative root system was of rank $\ge 2$. In our current case, such a parabolic subgroup exists already over $A$.
\end{proof}

\begin{lem}\label{lem:PS-16'}
Fix $s\in A$. For any $g(X)\in E(A_s[X])$ there exists $k\ge 0$ such that $g(aX)g(bX)^{-1}\in F_s(E(A[X]))$
for any $a,b\in A$ satisfying $a\equiv b\pmod{s^k}$.
\end{lem}
\begin{proof}
Consider the element $f(Z)=g(X(Y+Z))g(XY)^{-1}$ of the group $E(A_s[X,Y,Z])$, where $X,Y,Z$ are three variables.
Then $f(0)=1$, so
$$
f(Z)\in E(A_s[X,Y,Z],ZA_s[X,Y,Z]).
$$
By Lemma~\ref{lem:PS-15'}
there exist $h(Z)\in E(A[X,Y,Z],ZA[X,Y,Z])$ and $k\ge 0$ such that $F_s\l(h(Z)\r)=f(s^kZ)$. We have
$$
f(s^kZ)=g\l(X(Y+s^kZ)\r)g(XY)^{-1}.
$$
If $a-b=s^kt$, $t\in A$, then setting $Y=b$, $Z=t$, we deduce the claim of the Lemma.
\end{proof}

\begin{proof}[Proof of Lemma~\ref{lem:A_fg}]
By~\cite[Lemma 8]{PS} we can find such $x(X)\in E(A_{fg}[X])$ that $x(0)=1$ and $x(1)=x$.
Then it is enough to find
$x_1(X)\in E(A_f[X])$ and $x_2(X)\in E(A_g[X])$ such that $x(X)=x_1(X)x_2(X)$. Since $x(0)=1$,
by Lemma~\ref{lem:PS-16'} there exists such $k\ge 0$ that for any $a,b\in A_{fg}$ such that
$a\equiv b\pmod {f^k}$, we have $x(aX)x(bX)^{-1}\in F_f(E(A_g[X]))$; and for any
$a,b\in A_{fg}$ such that
$a\equiv b\pmod {g^k}$, we have $x(aX)x(bX)^{-1}\in F_g(E(A_f[X]))$. Since $fA+gA=A$, we have $f^kA+g^kA=A$ as well.
Hence $1=f^ks+g^kt$ for some $s,t\in A$.
Then we have
$$
x(X)=x\bigl((f^ks+g^kt)X\bigr)x(g^ktX)^{-1}x(g^ktX)x(0\cdot X)^{-1}.
$$
Therefore, $x\bigl((f^ks+g^kt)X\bigr)x(g^ktX)^{-1}$ belongs to $F_f(E(A_g[X]))$, and $x(g^ktX)x(0\cdot X)^{-1}$ belongs to $F_g(E(A_f[X]))$.
\end{proof}

\section{The functor $K_1^G$}



\subsection{Definition and basic properties of $K_1^G$}

Let $G$ be a reductive group over a commutative ring $A$ satisfying the condition \ee.

\begin{defn}
The functor $K_1^G$ on the category of commutative $A$-algebras $R$, given by $K_1^G(R)=G(R)/E(R)$,
is called the \emph{non-stable $K_1$-functor} associated to $G$.
\end{defn}
The normality of the elementary subgroup implies that $K_1^G$ is a group-valued functor.
Following Bass, define $NK_1^G(R)$ as the kernel of the map $K_1^G(R[X])\to K_1^G(R)$ induced by evaluation at
$x=0$, and denote by $\max(R)$ the maximal spectrum of $R$. Then by Suslin's local-global principle,
for any $A$-algebra $R$ the natural map
$$
NK_1^G(R)\xrightarrow{\prod F_m}\prod_{m\in\max(R)}NK_1^G(R_m)
$$
is injective.
In other words, $NK_1^G$ is a separated presheaf for the Zariski topology on the category of affine $A$-schemes.


\begin{defn} We say that $K_1^G$ \emph{is $\Aff^1$-invariant at $R$}, if the map $K_1^G(R)\to K_1^G(R[X])$ induced by
the inclusion of $R$ into $R[X]$ is an isomorphism, or, equivalently, if
$$G(R[X])=G(R)E(R[X]).$$
\end{defn}

 In Theorem~\ref{th:whitehead-geometric}
we show that $K_1^G$ is $\Aff^1$-invariant at $R$, if $G$
is a reductive algebraic group over a perfect field $k$ that satisfies \ee, and $R$ is a regular $k$-algebra.

\smallskip
{\bf Remark.} Note that if a reductive group $G$ over a commutative ring $R$ is isotropic, i.e. contains a proper parabolic subgroup $P$
(it is reasonable to assume that $P$ is strictly proper), but $G$ does not necessarily satisfy \ee, one can still consider
the quotient
$$
K_1^{G,P}(R)=G(R)/E_P(R).
$$
If $R$ is a semilocal ring, one knows that $E_P(R)$ is also independent of the choice of $P$ and normal in
$G(R)$ by Theorem~\ref{th:PS-normality} (i). However, if $R$ is not semilocal, $E_P(R)$ is not in general normal
in $G(R)$, and Suslin's local-global principle is not true.
Also, the classical example of Cohn~\cite{Cohn} says that $SL_2(k[X_1,X_2])\neq E_2(k[X_1,X_2])$. Since
$SL_2(k[X_1])=E_2(k[X_1])$, this implies that $K_1^{\SL_2,P_1}$ is not $\Aff^1$-invariant at $k[X_1]$.
One may ask if the subgroup $\hat E(R)$ of $G(R)$ generated by all $E_P(R)$, $P$ a parabolic subgroup of $G$, provides a better definition
of $K_1^G$ if \ee\ is not satisfied. Unfortunately, if $k$ is a finite field, then again
$\SL_2(k[X_1,X_2])\neq \hat E_2(k[X_1,X_2])$, see~\cite[Th. 1.4]{GMV-SL2} and~\cite{KrMc}.

These are the reasons why in the present paper we mostly restrict our attention to groups of isotropic rank $\ge 2$.

\subsection{Margaux---Soul\'e theorem.}

The following theorem is an essential ingredient in the proof of our main theorem, Theorem~\ref{th:whitehead-geometric}.

\begin{thm}[Margaux--Soul\'e]\label{th:MS-red}
Let $G$ be a reductive algebraic group over a field $k$, such that every normal semisimple subgroup of the algebraic derived
group $[G,G]$ is isotropic.
Then
$$
\begin{array}{rl}
G(k[X])&=G(k)\cdot\l<U_P(k[X]),\ P\mbox{ a minimal parabolic $k$-subgroup of }G\r>\\
&= G(k)\cdot E_Q(k[X])
\end{array}
$$ for any strictly proper parabolic $k$-subgroup $Q$ of $G$.
\end{thm}

The second equality of this theorem is equivalent to the first by Theorem~\ref{th:PS-normality} (i).
The first equality was obtained by B. Margaux~\cite[Corollary 3.6]{M} for any isotropic simply connected absolutely
almost simple group $G$ over a field $k$, extending the respective result of C. Soul\'e for Chevalley groups~\cite{Sou}.
This statement in the above papers was, actually, a corollary of a more general
result about buildings associated to such groups $G$; we will not need it in full generality.


The following standard lemma allows to extend Margaux's result to isotropic reductive groups.

\begin{lem}\label{lem:sc-red}
Let $G$ be a reductive algebraic group over a field $k$.
Let $G^{\scl}$ be the simply connected semisimple group
isogenous to the algebraic derived subgroup $[G,G]$ of $G$. Let $A$ be a normal Noetherian integral domain containing $k$.
If one has
$$
G^{\scl}(A[X])=G^{\scl}(A)E_P(A[X])
$$
for some minimal parabolic $k$-subgroup $P$ of $G^{\scl}$,
then
$$
G(A[X])=G(A)E_Q(A[X])
$$
for any minimal parabolic $k$-subgroup $Q$ of $G$.
\end{lem}
\begin{proof}
Let $Q$ be a minimal parabolic $k$-subgroup of $G$.
It is enough to show that any element $g\in \ker\bigl(G(A[X])\xrightarrow{X\mapsto 0}G(A)\bigr)$ belongs to $E_Q(A[X])$.

First, let $G$ be a possibly non-simply connected semisimple group over $k$, satisfying the conditions of the Lemma.
There is a short exact sequence of algebraic groups
$$
1\to C\xrightarrow{i} G^{\scl}\xrightarrow{\pi} G\to 1,
$$
where $C$ is a group of multiplicative type over $k$, central in $G^{\scl}$. Write the respective ``long'' exact sequences over $A[X]$ and $A$
with respect to fppf topology. Adding the maps induced by the homomorphism $\rho:A[X]\to A$, $X\mapsto 0$,
we obtain a commutative diagram
\begin{equation*}
\xymatrix@R=15pt@C=20pt{
1\ar[r]\ar@{=}[d]&C(A[X])\ar[d]^\rho\ar[r]^{i}&G^{\scl}(A[X])\ar[d]^\rho\ar[r]^\pi &G(A[X])\ar[d]^\rho\ar[r]^{\hspace{-15pt}\delta}&H^1_{\fppf}(A[X],C)\ar[d]^{\cong}\\
1\ar[r]&C(A)\ar[r]^{i}&G^{\scl}(A)\ar[r]^\pi&G(A)\ar[r]^{\hspace{-15pt}\delta}&H^1_{\fppf}(A,C)\\
}
\end{equation*}
Here the rightmost vertical arrow is an isomorphism by~\cite[Lemma 2.4]{CTS}. Take any
$g\in\ker\bigl(\rho:G(A[X])\to G(A)\bigr)$. Then $\delta(g)=1$, hence there is $\tilde g\in G^{\scl}(A[X])$ with
$\pi(\tilde g)=g$. Clearly, $\rho(\tilde g)\in C(A)$, and hence
$$
\tilde g\in C(A)\cdot \ker\bigl(\rho:G^{\scl}(A[X])\to G(A)\bigr)\subseteq C(A)E_P(A[X]).
$$
Since
$\pi(E_P(A[X]))=E_Q(A[X])$ by Theorem~\ref{th:PS-normality} (i), this proves the claim for $G$.

Now let $G$ be any reductive group over $k$ satisfying the conditions of the Lemma. Then
there is a short exact sequence
$$
1\to [G,G]\to G\to T\to 1,
$$
for a $k$-torus $T$. Here the group $[G,G]$ is the algebraic derived subgroup of $G$; it is a semisimple group that
satisfies the isotropy conditions of the Lemma, if $G$ does. Moreover, it contains the unipotent radicals of all parabolic subgroups
of $G$, hence the subgroups $E_Q(A[X])$ are the same for $[G,G]$ and $G$. Since
$T(A[X])\cong T(A)$ (e.g. by \'etale descent),
the exact sequence
$$
1\to [G,G](A[X])\to G(A[X])\to T(A[X])
$$
finishes the proof.
\end{proof}

\begin{proof}[Proof of Theorem~\ref{th:MS-red}]
By Lemma~\ref{lem:sc-red}
it is enough to show that the claim holds if $G$ is a simply connected semisimple group.
By the results of~\cite[Exp. XXIV Prop. 5.10 and \S 5.3]{SGA},
any such $G$ is isomorphic to a finite direct product of simply connected semisimple $k$-groups of the form
$R_{k'/k}(H)$, where $k'$ is a finite separable field extension of $k$, $R_{k'/k}$ denotes the Weil
restriction functor, and $H$ is a simply connected absolutely almost simple group over $k'$. Clearly, any $H$ involved
in the decomposition of $G$ has to be isotropic, and any minimal parabolic subgroup of $R_{k'/k}(H)$
is a Weil restriction $R_{k'/k}(Q)$ of a minimal parabolic subgroup $Q$ of $H$. Also, $U_{R_{k'/k}(Q)}=R_{k'/k}(U_Q)$.
By~\cite[Corollary 3.6]{M} one has
$$
H(k'[X])=H(k')\cdot\l<U_Q(k'[X]),\ Q\mbox{ a minimal parabolic $k'$-subgroup of }H\r>.
$$
By Theorem~\ref{th:PS-normality} (i),
it is enough to consider just one pair of opposite minimal parabolic subgroups $Q$ and $Q^-$ of $H$.
Then we have
$$
\begin{array}{rl}
R_{k'/k}(H)(k[X])&=H(k'[X])=H(k')\l<U_Q(k'[X]),U_{Q^-}(k'[X])\r>\\
&=R_{k'/k}(H)(k)\l<U_{R_{k'/k}(Q)}(k[X]),U_{R_{k'/k}(Q^-)}(k[X])\r>\\
&=R_{k'/k}(H)(k)\cdot E_{R_{k'/k}(Q)}(k[X]).
\end{array}
$$
This implies the claim of the theorem.
\end{proof}


\subsection{Relation to the Karoubi--Villamayor $K$-theory.}

For any reductive group $G$ over a commutative ring $A$, let $KV_1^G(A)$ denote the first Karoubi--Villamayor
$K$-group of the functor $G$, as defined by Jardine in~\cite[\S 3]{J} following Gersten~\cite{Ge}.
Note that Jardine denotes the Karoubi---Villamayor $K$-functor
by $K_1^G$, while we reserve this notation for our $K_1$-functor. The following result
is a straightforward extension to isotropic reductive groups of~\cite[Lemma 2.4]{W10} for Chevalley groups.
Note that even for Chevalley groups, the groups $K_1^G(A)$ are in general non-abelian (\cite{vdK}, see also~\cite{BHV}).

\begin{lem}\label{lem:KV}
Let $G$ be an isotropic reductive group over a unital commutative ring $A$ satisfying~\ee. There is an
exact sequence (a coequalizer)
$$
K_1^G(A[X])\xrightarrow{g\mapsto g(1)g(0)^{-1}}K_1^G(A)\to KV_1^G(A)\to 1,
$$
where the first map is a map of pointed sets, while the second one is a group homomorphism.

In particular, if $K_1^G$ is $\Aff^1$-invariant at $A$, then $K_1^G(A)\cong KV_1^G(A)$ as groups.
\end{lem}
\begin{proof}
We first recall the necessary notation from~\cite{J}.
Let $p$ denote both  maps $A[X]\to A$ and $G(A[X])\to G(A)$ induced by $X\mapsto 0$, and $\eps$
denote both maps $A[X]\to A$ and $G(A[X])\to G(A)$ induced by $X\mapsto 1$. Set
$EA=\ker(p:A[X]\to A)$. Let $\tilde G$ be the extension of the functor $G$ to the category of not necessarily unital
commutative $A$-algebras $R$, defined by
$$\tilde G(R)=\ker\bigl(pr_A:G(A\oplus R)\to G(A)\bigr).
$$
Here the unital algebra $A\oplus R$ is the direct sum of additive groups with multiplication given
by $(\alpha,a)\cdot(\beta, b)=(\alpha\beta, \alpha b+\beta a+ab)$. By definition,
$$
KV_1^G(A)=\coker\bigl(\eps:\tilde G(EA)\to\tilde G(A)\bigr).
$$

Now we establish the existence of a surjective homomorphism $K_1^G(A)\to KV_1^G(A)$. Thus, there is a canonical group homomorphism
$G(A)\cong \tilde G(A)\to KV_1^G(A)$. We have $E(A)\subseteq \eps(\tilde G(EA))$, where $\tilde G(EA)$ is identified with
its image in $\tilde G(A)$. Indeed, $\tilde G(EA)=\ker\bigl(G(A\oplus EA)\to G(A)\bigr)$; we have $A\oplus EA\cong A[X]$,
hence $\tilde G(EA)$ is identified with the kernel of $p:G(A[X])\to G(A)$. By~\cite[Lemma 8]{PS} for any $g\in E(A)$ there is
$g(X)\in E(A[X])\subseteq G(A[X])$ such that $g(0)=1$ and $g(1)=g$. Hence $E(A)\subseteq \eps\bigl(\ker(G(A[X])\to G(A)\bigr)$.
Summing up, there is a correctly defined map $K_1^G(A)=G(A)/E(A)\to  KV_1^G(A)$. Clearly, it is surjective.

Now we show the exactness at the $K_1^G(A)$ term. By~\cite[Lemma 3.5]{J} the inclusion $A\to A[X]$
 induces an isomorphism between $KV_1^G(A)$ and $KV_1^G(A[X])$. Consider the image of $g(1)g(0)^{-1}\in K_1^G(A)$
in $K_1^G(A[X])$ under the inclusion map. One readily sees that
$$
g(1)g(0)^{-1}=\l(g(Y)g(0)^{-1}\r)|_{Y=1}\in \eps_Y\bigl(ker\l(p_Y:G(A[X,Y])\to G(A[X])\r)\bigr),$$
 where $\eps_Y$, $p_Y$ are the same as $\eps$, $p$ with respect to the free
variable $Y$. Therefore, the image of $g(1)g(0)^{-1}$ in $KV_1^G(A[X])$ is trivial, which implies that
it is in $\ker(K_1^G(A)\to KV_1^G(A))$. Now let $g\in G(A)$ be an element whose image under
$$
G(A)\to K_1^G(A)\to KV_1^G(A)
$$ is trivial. Then $g$ belongs to $\eps\bigl(\ker (p:G(A[X])\to A)\bigr)$. This means that there is $g(X)\in G(A[X])$
such that $g(0)=1$ and $g(1)=g$. Then $g=g(1)g(0)^{-1}$ belongs to the image of the map
$K_1^G(A[X])\to K_1^G(A)$ in our exact sequence.
\end{proof}

\subsection{Nisnevich gluing for $K_1^G$.}

We prove here a gluing property of $K_1^G$, 
which looks like a segment of the Mayer-Vietoris sequence for a distinguished Nisnevich square.
It is a straightforward extension
of~\cite[Lemma 2.4]{Vo} and~\cite[Lemma 3.7]{Abe} for split groups; we only replace the usual split root subgroups
by relative root subschemes. 

\begin{lem}\label{lem:Abe-3.7}
Let $G$ be an isotropic reductive group over a commutative ring $B$.
Let $P$ be a strictly proper
parabolic subgroup of $G$, such that all irreducible components of the relative root system $\Phi_P$
are of rank $\ge 2$ everywhere on $\Spec B$.
Assume moreover that $B$ is a subring of a commutative ring $A$, and let
$h\in B$ be a non-nilpotent element. Denote by $F_h:G(A)\to G(A_h)$ the localization homomorphism.

(i) If $Ah+B=A$, i.e. the natural map $B\to A/Ah$ is surjective, then for any $x\in E(A_h)$ there exist $y\in E(A)$ and $z\in E(B_h)$ such that
$x=F_h(y)z$.

(ii) If moreover $Ah\cap B=Bh$, i.e. $B/Bh\to A/Ah$ is an isomorphism, and $h$ is not a zero divizor in $A$, then
the sequence of pointed sets
$$
K_1^G(B)\xrightarrow{\st g\mapsto (F_h(g),g)} K_1^G(B_h)\times K_1^G(A)\xrightarrow{\st (g_1,g_2)\mapsto g_1F_h(g_2)^{-1}}
K_1^G(A_h)
$$
is exact.
\end{lem}
\begin{proof}

(i) One has $x=\prod\limits_{i=1}^mX_{\beta_i}(c_i)$, $c_i\in A_h\otimes_k V_{\beta_i}$, $\beta_i\in\Phi_P$.
 We need to show that $x\in F_h(E(A))E(B_h)$. Clearly, it is enough to show that
\begin{equation}\label{eq:EBh}
E(B_h)X_\beta(c)\subseteq F_h(E(A))E(B_h)
\end{equation}
for any $\beta\in\Phi_P$ and $c\in V_\beta\otimes_B A_h$.
We can assume that $\beta$ is a positive relative root without loss of generality.
We prove the inclusion~\eqref{eq:EBh} by descending induction on the height of $\beta$.

Take $z\in E(B_h)$.  By Lemma~\ref{lem:PS-15'}, there exist $N\ge 0$ and $y(Z)\in E(A[Z],ZA[Z])$ such that
$F_h(y(Z))=zX_{\beta}(h^NZ)z^{-1}$.
On the other hand, note that $Ah+B=A$ implies $Ah^n+B=A$ for any $n\ge 1$. Let $M\ge 0$ be such that
$h^Mc\in V_\beta\otimes_B A$.
Then one can find
$a\in V_{\beta}\otimes_B A$ and $b\in V_{\beta}$ such that
$$
c=ah^N+h^{-M}b.
$$
By the multiplication formula for relative root elements~\eqref{eq:sum} we have
$$
X_{\beta}(c)=X_{\beta}(ah^{N})X_{\beta}(h^{-M}b)\prod\limits_{i\ge 2}X_{i\beta}\l(q^i_\beta(h^Na,h^{-M}b)\r).
$$
By the choice of $N$, one has
$$
zX_{\beta}(ah^{N})X_{\beta}(h^{-M}b)=\l(zX_{\beta}(ah^{N})z^{-1}\r)zX_{\beta}(h^{-M}b)\in F_h\l(E(A)\r)E(B_h).
$$
Since the height of the relative roots $i\beta$, $i\ge 2$, is larger than that of $\beta$, the inductive
hypothesis implies $zX_\beta(c)\in F_h(E(A))E(B_h)$. This proves~\eqref{eq:EBh}.

(ii)
Take $x_1\in G(B_h)$, $x_2\in G(A)$ such that $x_1F_h(x_2)^{-1}\in E(A_h)$. Then by~(i) we have
$y_1x_1=F_h(y_2x_2)$ for some $y_1\in E(B_h)$, $y_2\in E(A)$.
By assumption, $Ah^n\cap B=Bh^n$ for any $n\ge 0$, and hence $F_h(A)\cap B_h=F_h(B)$ in $A_h$.
Since $F_h$ is injective, we have $y_2x_2\in G(B)$. The claim follows.
\end{proof}

\begin{cor}\label{cor:Nis-square}
Let $G$ be an isotropic reductive group over a commutative ring $B$. Let $P$ be a strictly proper
parabolic subgroup of $G$, such that all irreducible components of the relative root system $\Phi_P$
are of rank $\ge 2$ everywhere on $\Spec B$.

Let $\phi:B\to A$ be a homomorphism of commutative rings, and $h\in B$, $f\in A$ non-nilpotent
elements such that $\phi(h)\in fA^\times$ and $\phi:B/Bh\to A/Af$ is an isomorphism.
Assume moreover that the commutative square
\begin{equation*}
\xymatrix@R=15pt@C=20pt{
\Spec A_f\ar[r]^{F_f}\ar[d]^{\phi}&\Spec A\ar[d]^{\phi}\\
\Spec B_h\ar[r]^{F_h}&\Spec B\\
}
\end{equation*}
is a distinguished Nisnevich square.
Then the sequence of pointed sets
$$
K_1^G(B)\xrightarrow{\st (F_h,\phi)} K_1^G(B_h)\times K_1^G(A)\xrightarrow{\st (g_1,g_2)\mapsto \phi(g_1)F_f(g_2)^{-1}}K_1^G(A_f)
$$
is exact.
\end{cor}
\begin{proof}
Since $A_f=A_{\phi(h)}$, we can assume that $f=\phi(h)$ from the start.
We have the following commutative diagram.
\begin{equation*}
\xymatrix@R=15pt@C=20pt{
B\ar[r]^{\phi}\ar[d]^{F_h}&\phi(B)\ar[d]^{F_f}\ar@{^{(}->}[r]&A\ar[d]^{F_f}\\
B_h\ar[r]&\phi(B_h)=\phi(B)_f\ar@{^{(}->}[r]&A_f\\
}
\end{equation*}
Since $E(A_f)\subseteq F_f(E(A))E(\phi(B)_f)$ by Lemma~\ref{lem:Abe-3.7} (i), and
$\phi:E(B_h)\to E(\phi(B)_f)$ is surjective, we have
\begin{equation}\label{eq:EAf}
E(A_f)\subseteq F_f(E(A))\cdot \phi(E(B_h)).
\end{equation}
Take $x_1\in G(B_h)$, $x_2\in G(A)$ such that $\phi(x_1)F_f(x_2)^{-1}\in E(A_f)$. Then by~\eqref{eq:EAf} we have
$\phi(y_1x_1)=F_f(y_2x_2)$ for some $y_1\in E(B_h)$, $y_2\in E(A)$. Since $G$ is a sheaf in the Nisnevich topology,
there is $z\in G(B)$ such that $\phi(z)=y_2x_2$, $F_h(z)=y_1x_1$. This implies the claim of the Lemma.
\end{proof}

\section{Supplementary lemmas}

In this section we prove a few more lemmas on the structure of the elementary subgroup of a
reductive group. We use the notation of \S\ref{ssec:rel}, in particular, Definition~\ref{def:E(A,I)}.

\subsection{Elementary subgroup over a polynomial ring}

The following lemma extends~\cite[Prop. 1.6, Prop. 1.8, Cor. 1.9]{Abe}.

\begin{lem}\label{lem:Abe-sect}
Let $A$ be a commutative ring, and let $I$ be an ideal of $A$ such that the projection $\pi:A\to A/I$ has a section
$i:A/I\to A$, i.e. $i$ is a homomorphism such that $\pi\circ i=\id$. Set $B=i(A/I)\subseteq A$.
Let $G$ be a reductive group scheme over $A$, and let $P$ be a proper parabolic subgroup of $G$.
Then

(i) $E_P^*(A,I)=E_P(A,I)=E_P(I)^{E_P(B)}$; in particular,
$$
E_P^*(A[X],XA[X])=E_P(A[X],XA[X]).
$$

(ii) $E_P(A)\cap G(B)=E_P(B)$.
\end{lem}
\begin{proof}
We can assume that the relative roots and root subschemes with respect to $P$ are correctly defined
over $A$, and hence over $B$. Let $\Phi_P$ be the relative root
system of $G$ with respect to $P$ over $A$.

To prove (i), assume that $g=\prod\limits_{i=1}^mX_{\beta_i}(u_i)\in E_P^*(A,I)$, where $\beta_i\in\Phi_P$, $u_i\in V_{\beta_i}$.
Here $V_{\beta_i}$ are the respective finitely generated projective $A$-modules. Since
$$
\ker(\pi:V_{\beta_i}\to V_{\beta_i}\otimes_A A/I)=V_{\beta_i}\otimes_A I,
$$
we have $u_i=t_i+v_i$, where $t_i=i(\pi(u_i))\in V_{\beta_i}\otimes_A B$ and
$v_i=u_i-t_i\in V_{\beta_i}\otimes_A I\subseteq V_{\beta_i}$. By~\eqref{eq:sum}, we have
$$
X_{\beta_i}(u_i)=X_{\beta_i}(t_i)X_{\beta_i}(v_i)\prod_{k>1}X_{k\beta_i}(w_{i,k}),
$$
where each $w_{i,k}=q_{\beta_i}^k(t_i,v_i)\in V_{k\beta_i}\otimes_A I$, since $q^k_{\beta_i}$ is
a homogeneous map of degree $k>0$. Therefore, $X_{\beta_i}(u_i)=X_{\beta_i}(t_i)h_i$, for some $h_i\in E_P(I)$.
Set $g_k=\prod\limits_{i=1}^kX_{\beta_i}(t_i)$, $1\le k\le m$. We have $g_m=1$, since $\pi(g)=1$. Then
$$
g=\prod_{k=1}^mg_k h_k g_k^{-1}\in E_P(I)^{E_P(B)}\subseteq E_P(A,I).
$$
This settles (i).

To prove (ii), note that $\pi:E_P(A)\to E_P(B)$ is surjective, hence $E_P(A)=E_P(B)E_P^*(A,I)$. Since
$G(B)\cap G(A,I)$ is trivial, we deduce $E_P(A)\cap G(B)=E_P(A)$.
\end{proof}

The following lemma extends~\cite[Lemma 3.6]{Abe} and~\cite[Lemma 2.1]{Vo}.

\begin{lem}\label{lem:Abe-3.6}
Let $A$ be a commutative ring, $S$ a multiplicative subset of $A$. Let $G$ be a reductive group scheme over $A$,
and P a proper parabolic subgroup of $G$. If
$G(A[X_1,\ldots,X_n])=G(A)E_P(A[X_1,\ldots,X_n])$ for some $n\ge 1$, then
$G(A_S[X_1,\ldots,X_n])=G(A_S)E_P(A_S[X_1,\ldots,X_n])$ as well.
\end{lem}
\begin{proof}
Exactly as~\cite[Lemma 3.6]{Abe}, using the relative root subschemes with respect to $P$ instead
of the usual root subgroups of Chevalley groups.
\end{proof}

\subsection{Generators of the congruence subgroup $E(A,I)$}

Let $A$ be any commutative ring, $G$ an isotropic reductive group over $A$, $P$ a parabolic subgroup of $G$.
We assume that the system of relative roots $\Phi_P$ and the respective relative root subschemes
are defined over $A$.

Let $\alpha\in\Phi_P$ be a relative root, we will denote by $m_\alpha$ the positive integer satisfying
$\Phi_P\cap\ZZ\alpha=\{\pm\alpha,\pm2\alpha,\ldots,\pm m_\alpha\alpha\}$. For $a\in E_\alpha(A)$,
$u_i\in V_{i\alpha}$, $1\le i\le m_\alpha$, we define
$$
Z_\alpha(a,u_1,\ldots,u_{m_\alpha})=a\left(\prod_{i=1}^{m_\alpha} X_{i\alpha}(u_i)\right)a^{-1}.
$$

The following lemma extends~\cite[Prop. 1.4]{Abe}.

\begin{lem}\label{lem:Abe-z(a,u)}
Let $A$, $G$, $P$ be as above. For any ideal $I$ of $A$, the group $E_P(A,I)$ is generated by $Z_\alpha(a,u_1,\ldots,u_{m_\alpha})$ for all
$\alpha\in\Phi_P$,
$u_i\in IV_{i\alpha}$ $(1\le i\le m_\alpha)$, and $a\in E_\alpha(A)$.
\end{lem}
\begin{proof}
Take any $\beta\in\Phi_P$, $c\in V_\beta$, and $\alpha\in\Phi_P$. It is enough to show that for any
$a\in E_\alpha(A)$, $u_i\in IV_{i\alpha}$, $1\le i\le m_\alpha$,
$$
x=X_\beta(c)Z_\alpha(a,u_1,\ldots,u_{m_\alpha})X_\beta(c)^{-1}
$$
is a product of elements $Z_\gamma(b,v_1,\ldots,v_{m_\gamma})$, where $\gamma\in\Phi_P$, $b\in E_\gamma(A)$, and $v_i\in IV_{i\gamma}$ for all
$1\le i\le m_\gamma$. If $\beta$ is collinear to $\alpha$, then, since both $\alpha$ and $\beta$ are integral
linear combinations of simple relative roots, there is a root $\gamma\in\Phi_P$ such that
$\alpha=k\gamma$, $\beta=l\gamma$, and $k,l\in\ZZ$. In this case we have
$x=Z_\gamma\l(X_\beta(c)a,v_1,\ldots,v_{m_\gamma}\r),$
where $v_i=u_{\frac ik}$ if $k$ divides $i$, and $v_i=0$ otherwise. If $\beta$ is non-collinear to $\alpha$, then
by Lemma~\ref{lem:Abe-1.5} below we have $x\in Z_\alpha(a,u_1,\ldots,u_{m_\alpha})\cdot E_P(I)$.
\end{proof}

\begin{lem}\label{lem:Abe-1.5}
Let $\alpha,\beta\in\Phi_P$ be two non-collinear relative roots, $I,J$ two ideals of $A$; let $A'$ be a
commutative $A$-algebra.
Let $a\in E_\alpha(A)$,
$t\in A'$, $u_i\in IV_{i\alpha}$  $(1\le i\le m_\alpha)$, and $v\in tJV_\beta\subseteq JV_\beta\otimes_A A'$.
Then
$$
X_\beta(v)Z_\alpha(a,u_1,\ldots,u_{m_\alpha})X_\beta(v)^{-1}=Z_\alpha(a,u_1,\ldots,u_{m_\alpha})y,
$$
where $y$ is a product of elements of the form $X_\gamma(w)$ with $\gamma=i\alpha+j\beta\in\Phi_P$, $i,j\in\ZZ$, $j>0$ and
$w\in t^jJ^jIV_\gamma\subseteq V_\gamma\otimes_A A'$.
\end{lem}
\begin{proof}
For any  $k\in\ZZ\setminus\{0\}$ and $w\in V_{k\alpha}$ we deduce, using the formula~\eqref{eq:sum}
and the generalized Chevalley commutator formula, that
$$
\begin{array}{l}
X_\beta(v)X_{k\alpha}(w)=X_{k\alpha}(w)[X_{k\alpha}(w)^{-1},
X_\beta(v)]X_\beta(v)\\
=X_{k\alpha}(w)\cdot\prod\limits_{i,j>0}X_{k i\alpha+j\beta}(v_{ij})\cdot X_\beta(v),\quad v_{ij}\in
t^jJ^jV_{k i\alpha+j\beta}.\\
\end{array}
$$
Moreover, if $w\in IV_{k\alpha}$, then also $v_{ij}\in t^jJ^jI^iV_{k i\alpha+j\beta}$.
Note that for any $k,k'\in\ZZ\setminus \{0\}$, $i\ge 0$, $j>0$,
$i'>0$, and $j'\ge 0$,  the roots $\delta=k i\alpha+j\beta$ and
$\delta'=k' i'\alpha+j'\beta$ do not satisfy $m\delta=-n\delta'$ for any $m,n\ge 1$, and hence
the pair of root subschemes $X_\delta$, $X_{\delta'}$ is subject to the generalized Chevalley commutator
formula. Moreover, positive linear combinations of $\delta$ and $\delta'$ lie in the
set $\ZZ\alpha+\NN\beta$, and hence the corresponding relative root subschemes are also subject to the
generalized Chevalley commutator formula.
This way, commuting such root subschemes several times, we deduce
$$
[a^{-1},X_\beta(v)]=\prod\limits_{
\begin{array}{c}
\vspace{-3pt}\st i\in\ZZ,\\
\st j>0
\end{array}}X_{i\alpha+j\beta}(w_{ij}),\quad w_{ij}\in
t^jJ^jV_{i\alpha+j\beta},
$$
as well as
$$
\l[\l(\prod\limits_{i=1}^{m_\alpha}X_{i\alpha}(u_i)\r)^{-1}\hspace{-3pt},\,X_\beta(v)\r]=\prod\limits_{\begin{array}{c}
\vspace{-3pt}\st i\in\ZZ,\\
\st j>0
\end{array}}X_{i\alpha+j\beta}(s_{ij}),\quad s_{ij}\in
t^jJ^jIV_{i\alpha+j\beta}.
$$
Then we have
$$
\begin{array}{l}
X_\beta(v)Z_\alpha(a,u_1,\ldots,u_{m_\alpha})X_\beta(v)^{-1}=X_\beta(v)a\cdot\prod\limits_{i=1}^{m_\alpha}X_{i\alpha}(u_i)\cdot a^{-1}X_\beta(v)^{-1}\\
=
a\l[a^{-1},X_\beta(v)\r]X_\beta(v)\cdot\prod\limits_{i=1}^{m_\alpha}X_{i\alpha}(u_i)\cdot X_\beta(v)^{-1}
\l[X_\beta(v),a^{-1}\r]a^{-1}\\
=
a[a^{-1},X_\beta(v)]\cdot \prod\limits_{i=1}^{m_\alpha}X_{i\alpha}(u_i)\cdot
\biggl[\l(\prod\limits_{i=1}^{m_\alpha}X_{i\alpha}(u_i)\r)^{-1}\hspace{-2pt},\,X_\beta(v)\biggr]\cdot
\l[a^{-1},X_\beta(v)\r]^{-1}a^{-1}\vspace{3pt}\\
=
a\cdot\prod\limits_{i=1}^{m_\alpha}X_{i\alpha}(u_i)\cdot\biggl[\l(\prod\limits_{i=1}^{m_\alpha}X_{i\alpha}(u_i)\r)^{-1}\hspace{-2pt},
\prod\limits_{\begin{array}{c}
\vspace{-3pt}\st i\in\ZZ,\\
\st j>0
\end{array}}\hspace{-5pt} X_{i\alpha+j\beta}(w_{ij})\biggr]\cdot \\
\quad\cdot\Bigl[\prod\limits_{\begin{array}{c}
\vspace{-3pt}\st i\in\ZZ,\\
\st j>0
\end{array}}\hspace{-5pt} X_{i\alpha+j\beta}(w_{ij}),
\prod\limits_{\begin{array}{c}
\vspace{-3pt}\st i\in\ZZ,\\
\st j>0
\end{array}}\hspace{-5pt} X_{i\alpha+j\beta}(s_{ij})\Bigr]
\cdot\prod\limits_{\begin{array}{c}
\vspace{-3pt}\st i\in\ZZ,\\
\st j>0
\end{array}}\hspace{-5pt} X_{i\alpha+j\beta}(s_{ij})\cdot a^{-1}\\
=Z_\alpha(a,u_1,\ldots,u_{m_\alpha}) axa^{-1},
\end{array}
$$
where $x=\prod\limits_{\begin{array}{c}
\vspace{-3pt}\st i\in\ZZ,\\
\st j>0
\end{array}}\hspace{-5pt}X_{i\alpha+j\beta}(r_{ij}),\quad r_{ij}\in
t^jJ^jIV_{i\alpha+j\beta}$.
Applying the generalized Chevalley commutator formula to $axa^{-1}=[a,x]x$, one deduces the claim of the lemma.
\end{proof}


The following lemma is an analogue of~\cite[Cor. 2.7]{Abe}.

\begin{lem}\label{lem:a-nea}
Let $A$, $G$, $P$ be as above. Let $I$ be an ideal of $A$. Let $\alpha\in\Phi_P$ be a non-divisible relative root (i.e. all relative roots collinear to $\alpha$
are its integral multiples).
Then any element $x\in E_P(A,I)$ can be presented as a product $x=x_1x_2$, where $x_1\in E_\alpha(A,I)$, and
$x_2$ is a product of elements
of the form $Z_{\beta}(a,u_1,\ldots,u_{m_\beta})$, where
$\beta$ is non-collinear to $\alpha$, $u_i\in IV_{i\beta}$, $1\le i\le m_\beta$, and $a\in E_\beta(A)$.
\end{lem}
\begin{proof}
Follows by induction
 from Lemmas~\ref{lem:Abe-z(a,u)} and~\ref{lem:Abe-1.5}.
\end{proof}

We write down one more easy lemma for future reference.

\begin{lem}\label{lem:E*(A,I)}
Let $A$ be a local ring, $I$ the maximal ideal of $A$.
For any isotropic reductive group $G$ over $A$ with
two opposite parabolic subgroups $P=P^+$ and $P^-$ defined over $A$, having the common Levi subgroup
$L_P=P^+\cap P^-$ and unipotent radicals $U_{P^\pm}$, we have
$$
G(A,I)=U_{P^+}(I)\cdot L_P(A,I)\cdot U_{P^-}(I)=U_{P^-}(I)\cdot L_P(A,I)\cdot U_{P^+}(I).
$$
In particular, $G(A,I)=E_P(A,I)L_P(A,I)$.
\end{lem}
\begin{proof}
Let $\rho:A\to A/I$ be the quotient map. The product $\Omega=U_{P^+}\times L_P\times U_{P^-}$ embeds into $G$ as an open subscheme via the multiplication
morphism, e.g.~\cite[Exp. XXVI, 4.3.6]{SGA}. If for $g\in G(A)$
one has $\rho(g)\in\Omega(A/I)$, then, since $I$ is the maximal ideal of the local ring $A$, we have
$g\in\Omega(A)=U_{P^+}(A)L_P(A)U_{P^-}(A)$. If, moreover, $\rho(g)=1$, then $g$ belongs to
$$
\l(\ker(\rho)\cap U_{P^+}(A)\r)\cdot \l(\strut\ker(\rho)\cap L_P(A)\r)\cdot \l(\ker(\rho)\cap U_{P^-}(A)\r)=
U_{P^+}(I)L_P(A,I)U_{P^-}(I).
$$
Since $L_P$ normalizes $U_{P^\pm}$, we deduce the equality $G(A,I)=E_P(A,I)L_P(A,I)$.
\end{proof}

\subsection{Construction of certain automorphisms.}

We will use later the following two kinds of automorphisms of an isotropic reductive group over a
connected semilocal ring. The proofs of their existence basically consist in putting together some observations from~\cite[Exp. XXVI \S 7]{SGA}.

\begin{lem}\label{lem:sigma-def}
Let $G$ be a reductive group scheme over a connected semilocal ring $R$, $P$ a minimal proper parabolic subgroup of $G$, $L$
a Levi subgroup of $P$, $\Phi_P$ the respective system of relative roots, $\alpha_i\in\Phi_P$ a simple relative root.
Let $G^{\ad}=G/\Cent(G)\subseteq\Aut(G)$ be the corresponding adjoint group.

There is a closed embedding $\chi:\Gm\to G^{\ad}$, $t\mapsto\chi_t$, over $R$, such that for any $R$-algebra $R'$, and any
$t\in\Gm(R')=(R')^\times$, one has

1) $\chi_t(X_\alpha(u))=X_\alpha(t^{m_i(\alpha)}\cdot u)$ for any $\alpha\in\Phi_P$, $u\in V_\alpha\otimes_R R'$;

2) $\chi_t|_L=\id_L$.
\end{lem}
\begin{proof}
Let $P'$, $L'$ be the parabolic subgroup and the Levi subgroup of $G^\ad$ which are the images of $P$, $L$
under $G\to G^\ad$. Denote $\Phi_P$ by $\Psi$.
We use the results of~\cite[Exp. XXVI \S 7]{SGA}. Let $S$ be the maximal split subtorus of $L'$, $\Lambda=\Hom(S,\Gm)$,
$\Lambda^*=\Hom(\Gm,S)$. By~\cite[Exp. XXVI Th. 7.4]{SGA} there is a map $\Psi\to \Lambda^*$
which defines a root datum $(\Lambda,\Lambda^*,\Psi,\Psi^*)$. Since $G^\ad$ is a group of adjoint type,
by the compatibility with the absolute root datum~\cite[Exp. XXVI Th. 7.13]{SGA} the lattice $\Lambda$ is generated by
$\Psi$. For any $t\in R'$, let $\chi_t\in \Lambda^*$ be such that $\chi_t(\alpha_i)=t$ and $\chi_t(\alpha)=1$ for any
other simple root $\alpha$ of $\Psi$. Then $\chi:\Gm\to S$, $t\mapsto\chi_t$, is the desired embedding.
The action on the relative root
subschemes $X_\alpha$ is the desired one by~\cite[Th. 2]{PS}.
\end{proof}

\begin{lem}\label{lem:n_alpha}
Let $G$ be a reductive group scheme over a connected semilocal ring $R$, $P=P^+$ a minimal proper parabolic subgroup of $G$, $L$
a Levi subgroup of $P$, $P^-$ the corresponding opposite parabolic subgroup, $U_P$ and $U_P^-$ the unipotent radicals of $P$ and $P^-$.
There exists an element $n_P\in E_P(R)$ such that
$$
n_PLn_P^{-1}=L,\qquad n_PU_Pn_P^{-1}=U_{P^-},\qquad n_PU_{P^-}n_P^{-1}=U_{P}.
$$
\end{lem}
\begin{proof}
Let $S\subseteq L$ be a maximal split subtorus of $G$. By~\cite[Exp. XXVI, 7.2]{SGA} the subgroups
$P$ and $P^-$ are conjugate by an element $n_P\in \Norm_G(L)(R)=\Norm_G(S)(R)$, hence $n_PLn_P^{-1}=L$ and $n_PU_Pn_P^{-1}=U_{P^-}$.
Since the characters of $S$ on $\Lie(P^-)$ are opposite to those on $\Lie(P)$, we also have
$n_PU_{P^-}n_P^{-1}=U_P$. Since one has the Gauss decomposition $G(R)=U_P(R)U_{P^-}(R)U_P(R)L(R)$ by~\cite[Exp. XXVI, Th. 5.1]{SGA},
we can assume that $n_P\in E_P(R)$.
\end{proof}

\section{Decomposition of the elementary subgroup over Laurent polynomials}\label{sec:prove-xx}

\subsection{Decomposition theorem and a corollary}
The goal of this section is to prove the following theorem.

\begin{thm}\label{th:E+-+}
Let $G$ be an isotropic reductive algebraic group over a local ring $A$, satisfying the condition \ee.
Then
$$
E(A[X,X^{-1}])=E(A[X])E(A[X^{-1}])E(A[X]).
$$
\end{thm}

In the next section we deduce the main results of the paper from the following corollary of Theorem~\ref{th:E+-+}.

\begin{cor}\label{cor:XX^-1}
Under the assumptions of Theorem~\ref{th:E+-+}, let $I$ be the maximal ideal of $A$. Then
$$
E^*(A[X,X^{-1}],I\cdot A[X,X^{-1}])=E^*(A[X],I\cdot A[X])E^*(A[X^{-1}],I\cdot A[X^{-1}]).
$$
\end{cor}

\begin{proof}[Proof of Theorem~\ref{th:E+-+}.]
Lemma~\ref{lem:EXX-1} below provides the claim of Theorem~\ref{th:E+-+} in case where $G$ is an absolutely almost simple reductive group.
To prove the general case, we can right away assume that $G$ is semisimple, since the elementary subgroup
is contained in the derived subgroup of $G$. Clearly, we can also assume that $G$ is simply connected.
Then the group $G$ is a direct product of Weil restrictions of absolutely almost simple groups defined over finite
\'etale extensions of $R$ by~\cite[Exp. XIV Prop. 5.10]{SGA}. This implies the claim.
\end{proof}

\begin{proof}[Proof of Corollary~\ref{cor:XX^-1}.]
Lemma~\ref{lem:M*} below provides the claim of Corollary~\ref{cor:XX^-1} for $G$ absolutely almost simple. In general,
as in the proof of Theorem~\ref{th:E+-+}, we can assume that $G$ is a semisimple reductive group, and the
claim holds for the simply connected covering $G^{\scl}$ of $G$. Any $g\in E^*(A[X,X^{-1}],I\cdot A[X,X^{-1}])$ can
be lifted to an element $\tilde g\in E^{\scl}(A[X,X^{-1}])$, where $E^{\scl}(-)$ denotes the elementary
subgroup of $G^{\scl}(-)$. Then
$$
\rho(\tilde g)\in \Cent(G^{\scl})(l[X,X^{-1}])\cap E^{\scl}(l[X,X^{-1}]),
$$
where $\rho:A\to A/I=l$
is the residue homomorphism. Since $\Cent(G^{\scl})$ is an \'etale twisted form of a direct product of groups of
the form $\mu_n$, $n\ge 2$, one has $\Cent(G^{\scl})(l[X,X^{-1}])=\Cent(G^{\scl})(l)$. Therefore,
$\rho(\tilde g)$ belongs to $E^{\scl}(l)$, and lifts to an element of $E^{\scl}(A)$. Hence we can
assume that
$$
\tilde g\in E^{\scl}(A){E^{\scl}}^*(A[X,X^{-1}],I\cdot A[X,X^{-1}]).
$$
Then by Lemma~\ref{lem:M*} we have
$$
\tilde g\in E^{\scl}(A){E^{\scl}}^*(A[X],I\cdot A[X]){E^{\scl}}^*(A[X^{-1}],I\cdot A[X^{-1}]),
$$
which implies that $g$ belongs to ${E}^*(A[X],I\cdot A[X])\cdot {E}^*(A[X^{-1}],I\cdot A[X^{-1}])$,
as required.
\end{proof}

The rest of this section is devoted to the proof of Theorem~\ref{th:E+-+} and Corollary~\ref{cor:XX^-1}
for an absolutely almost simple group $G$, that is, of Lemma~\ref{lem:EXX-1} and Lemma~\ref{lem:M*}, respectively.

Let us sketch the proof of Lemma~\ref{lem:EXX-1}.
Let $\Psi=\Phi_P$ be the system of relative roots of $G$ with respect to a minimal parabolic subgroup $P$.
Set
$$
Z=E(A[X])E(A[X^{-1}])E(A[X]).
$$ To show that $E(A[X,X^{-1}])=Z$, we need to prove that
\begin{equation}\label{eq:XZ-claim}
X_\alpha(V_\alpha\otimes_A A[X,X^{-1}])Z\subseteq Z\quad\mbox{for any relative root $\alpha\in\Psi$.}
\end{equation}
By Lemma~\ref{lem:sigma-def} there exists an automorphism $\sigma$ of the group
$G(A[X,X^{-1}])$ such that
$$
\sigma(X_\alpha(u))=X_\alpha(X^{m_1(\alpha)}u)\quad\mbox{for any}\quad u\in V_\alpha\otimes_A A[X,X^{-1}].
$$
Here $m_1(\alpha)$ denotes the coefficient
of a suitably chosen simple relative root $\alpha_1$ in the decomposition of $\alpha$ into a sum
of simple roots. The inclusion
\eqref{eq:XZ-claim} can be proved by applying a high power of $\sigma$ or $\sigma^{-1}$, once we show
that $\sigma^{\pm 1}(Z)\subseteq Z$.

We show the inclusion $\sigma^{\pm 1}(Z)\subseteq Z$ in Lemma~\ref{lem:sigma-Z}, after several preliminary
steps. We write
$$
E(A[X])=E(A)E(A[X],XA[X]),\quad E(A[X^{-1}])=E(A)E(A[X^{-1}],X^{-1}A[X^{-1}]),
$$
and compute the images of factors under $\sigma$. Note that by symmetry between $X$ and $X^{-1}$,
the group $\sigma\bigl(E(A[X^{-1}],X^{-1}A[X^{-1}])\bigr)$ is the same as $\sigma^{-1}\bigl(E(A[X],XA[X])\bigr)$ with $X$ substituted by $X^{-1}$.

Lemma~\ref{lem:sigma(EX)} shows that $\sigma^{\pm 1}(E(A[X],XA[X]))$ is not too far from $E(A[X],XA[X])$,
namely,
$$
\sigma^{\pm 1}\l(E(A[X],XA[X])\r)\subseteq E(A[X])X_{\mp\ha}(X^{-1}V_{\mp\ha}),
$$
where $\ha$ is the positive relative root of maximal height. This  statement is proved in
subsection~\ref{ssec:sigma}, Lemmas~\ref{lem:sigma(EX)-1}--\ref{lem:sigma(EX)}. The proof
relies on an almost case-by-case study of the action of $\sigma$ on the groups
$E_\alpha(A[X])$,
performed in subsection \S\ref{ssec:E_alpha}, Lemmas~\ref{lem:Abe-tauX}--\ref{lem:BC-tauX}. In fact, instead of $\sigma$ we consider
auxiliary automorphisms $\tau_\alpha$ of the groups $G_\alpha(A[X,X^{-1}])$, which satisfy
$$
\sigma|_{E_\alpha(A[X])}=(\tau_\alpha)^{m_1(\alpha)}.
$$
This allows to work in suitable rank 2 root subsystems of $\Psi$ instead of the whole $\Psi$. The main
technical tool here is the generalized Chevalley commutator formula.

The computation of  $\sigma(E(A))$ is much easier and follows from the Gauss
decomposition in $G(A)$
with respect to the maximal parabolic subgroup $P_1$ corresponding to the simple root $\alpha_1$.
After that, we show that the factors $\sigma\bigl(E(A[X],XA[X])\bigr)$, $\sigma(E(A))$, and
$\sigma\bigl(E(A[X^{-1}],X^{-1}A[X^{-1}])\bigr)$ in the product
representing $\sigma(Z)$ can be reordered to the effect that $\sigma(Z)\subseteq Z$. This last step
just uses the fact that $E(A)$ normalizes $E(A[X],XA[X])$
and $E(A[X^{-1}],X^{-1}A[X^{-1}])$, and, in some cases, the technical Lemma~\ref{lem:ha-comm}.

Lemmas~\ref{lem:sigma-M} and~\ref{lem:M_+^*} use the automorphism $\sigma$ to obtain some properties of
the group $E^*(A[X],I\cdot A[X])$ necessary to deduce Lemma~\ref{lem:M*} from Lemma~\ref{lem:EXX-1}.

\subsection{The setting.}

From now on, we fix the following notation. Let $A$ be a {\bf local} ring  with the maximal ideal $I$
and residue field $l=A/I$, and let $\rho:A\to l$ be the natural map.
Let $G$ be a absolutely almost simple group scheme over $A$, satisfying \ee.

 $P=P^+$ a minimal parabolic subgroup of $G$, $P^-$ an opposite subgroup, $L=P^+\cap P^-$ their common Levi subgroup,
$U^\pm$ their unipotent radicals. By~\cite[\S 7]{SGA}, there is a maximal split subtorus $S$ of $G$
such that $L=\Cent_G(S)$.

Let $\Phi$ be the absolute root system of $G$, and let $\Psi=\Phi_P$ be the root system with respect to $P$, $S$. We consider
relative root subschemes $X_\alpha(V_\alpha)$, $\alpha\in\Psi$, defined as in~\cite{PS}.
Let $\Psi'$ be the set of non-multipliable roots in $\Psi$, i.e. the set of such $\alpha\in\Psi$ that $2\alpha\not\in\Psi$.


Let $\Pi=\{\alpha_1,\ldots,\alpha_n\}$ be a system of simple roots of $\Psi$. We write
$\alpha=\sum\limits_{i=1}^n m_i(\alpha)\alpha_i$, where
$m_i(\alpha)\in \ZZ$, for any $\alpha\in\Psi$. We denote by $\ha$ the highest positive root of $\Psi$. We assume that the numbering of $\Pi$ is chosen so that
$\alpha_1$ is a terminal vertex on the Dynkin diagram of $\Psi$, and $m_1(\ha)=1$; or, if such a vertex does not
exist, $m_1(\ha)=2$ and $\alpha_1$ is the unique
root adjacent to $-\ha$ in the extended Dynkin diagram of $\Psi$. Note that in the latter case $\ha$ is the only
positive root with $m_1(\ha)=2$; the respective standard maximal parabolic subgroup is called extraspecial.
If $\Psi$ has no multipliable roots, $\alpha_1$ is a long root; if $\Psi=BC_n$, then $\alpha_1$ is a root of middle
length (hence, non-multipliable), and $\{\alpha_1,\ldots,\alpha_{n-1},2\alpha_n\}$ is a system of simple roots for $\Psi'$.

We denote by $P_1^\pm$ the opposite standard maximal parabolic subgroups of $G$ corresponding to $\alpha_1$,
by $L_1$ their common Levi subgroup, and by $U_1^\pm$ their unipotent radicals.
Note that since $G$ satisfies \ee~over $A$, for any $A$-algebra $A'$ we
have
$$
E_P(A')=E_{P_1}(A')=E(A').
$$

Consider $G$ as a group over the ring of Laurent polynomials $A[X,X^{-1}]$.
\begin{defn}
Let $\sigma$ be a group automorphism of $G(A[X,X^{-1}])$ satisfying\\
\indent$\bullet$ $\sigma|_{L_1}=\id$;\\
\indent$\bullet$ $\sigma(X_\alpha(u))=X_\alpha(X^{m_1(\alpha)}u)$ for any $\alpha\in\Psi$, $u\in V_\alpha$.

Such an automorphism exists by Lemma~\ref{lem:sigma-def}.
\end{defn}


As in~\cite{Abe}, we denote
$$
\begin{array}{l}
M_+^*=E^*(A[X],I\cdot A[X])=G(A[X],I\cdot A[X])\cap E(A[X]),\\
M_-^*=E^*(A[X^{-1}],I\cdot A[X^{-1}])=G(A[X^{-1}],I\cdot A[X^{-1}])\cap E(A[X^{-1}]),\\
M^*=E^*(A[X,X^{-1}],I\cdot A[X,X^{-1}])=G(A[X,X^{-1}],I\cdot A[X,X^{-1}])\cap E(A[X,X^{-1}]).
\end{array}
$$

Recall that by Lemma~\ref{lem:Abe-sect} we have $E^*(A[X],XA[X])=E(A[X],XA[X])$.
By Lemmas~\ref{lem:Abe-z(a,u)} and~\ref{lem:Abe-sect}, this group is generated
by its subgroups $E_\alpha(A[X],XA[X])$, $\alpha\in\Psi$. The same results
also hold for $X^{-1}$ instead of $X$. From now on, we will use these facts without  a reference.

We will also use the following abbreviation:
$$
EL_1(A)=L_1(A)\cap E(A).
$$



\subsection{The automorphisms $\tau_\alpha$ and their action on $E_\alpha(A[X],XA[X])$.}\label{ssec:E_alpha}

\begin{lem}\label{lem:G_alpha}
Let $\alpha\in\Psi$ be a relative root. The group scheme $G$ contains an
$A$-subgroup $G_\alpha'$, which is either absolutely almost simple, or a Weil restriction of an absolutely almost simple group over a
finite \'etale extension
of $A$, such that the intersections $P^\pm\cap G_\alpha'$ are two opposite minimal parabolic subgroups of
$G_\alpha'$, and $U_{(\pm \alpha)}=\prod_{i\ge 1}X_{i\alpha}$ are the unipotent radicals of these subgroups.

If $\Psi=G_2,F_4,E_8$, and $\alpha$ is a long root, then there exists a root subsystem $\Theta_\alpha$
of $\Psi$ of type $A_2$ containing $\alpha$, and a split absolutely almost simple algebraic subgroup $G_{\Theta_\alpha}$ of $G$
over $A$ of type $A_2$, such that $X_\beta$, $\beta\in\Theta_\alpha$,
are root subgroups of $G_{\Theta_\alpha}$. In this case $G_\alpha'$ is the split absolutely almost simple subgroup
of $G_{\Theta_\alpha}$ of type $A_1$, corresponding to the root subsystem $\{\pm\alpha\}$.
\end{lem}
\begin{proof}
We will construct the desired subgroups in the case where $G$ is a simply connected absolutely almost simple group.
If $G$ is not simply connected, they are constructed as images of the respective subgroups
in the simply connected covering $G^{\scl}$ of $G$ under the natural homomorphism $G^{\scl}\to G$.

Recall that we have defined in section~\ref{ssec:rel} the reductive $A$-subgroup $G_\alpha$ of $G$. This subgroup
is determined by the fact that its Lie algebra has the form
$$
\Lie(G_\alpha)=\Lie(L)\oplus\bigoplus_{i\in\ZZ}\Lie(G)_{i\alpha},
$$
where $\Lie(G)_{i\alpha}$ is the submodule of $\Lie(G)$ corresponding to the character $i\alpha$ of $S$;
see~\cite[Prop. 6.1]{SGA}.
The derived subgroup $[G_\alpha,G_\alpha]$ of $G_\alpha$ is a semisimple group.

We can choose a basis $\Pi$ of simple roots in $\Psi$ so that $\alpha$ is either a simple root $\alpha_i$, or equals $-\ha$,
if $\Psi=BC_n$. In the first case one readily sees that $G_\alpha$ is the Levi subgroup
of the standard parabolic subgroup of $G$ corresponding to the set of simple roots
$\Pi\setminus\{\alpha\}$.  It is well-known that the
derived subgroup of a Levi subgroup of
a parabolic subgroup of $G$  is simply connected, if the ambient group $G$ is. Let $G_\alpha'$ be the semisimple
direct factor in
$[G_\alpha,G_\alpha]$, that contains $U_{(\pm\alpha)}$ and does not have proper normal semisimple subgroups.
In other words,  the Dynkin diagram of $G_\alpha'$ is the union of connected subdiagrams in the Dynkin diagram of $G_\alpha$,
containing roots projected onto the simple relative root $\alpha$. If there is only one such subdiagram,
$G_\alpha'$ is absolutely almost simple.
If $G$ is of outer type, the Dynkin diagram of $G_\alpha'$ may consist of several connected components permuted by
the $*$-action. Then $G_\alpha'$ is a Weil restriction of an absolutely almost simple group.

Moreover, if $\Psi=G_2,F_4,E_8$, the classification of Tits
indices~\cite{PS-tind} shows that for any long root $\alpha\in\Psi$ the module $V_\alpha$ is 1-dimensional, and,
clearly, there is a root subsystem $\Theta_\alpha\subseteq\Psi$ of type $A_2$ containing $\alpha$. Hence the
relative root subgroups $X_{\pm\alpha}$
are the usual root subgroups of the split group $G$ after a splitting base extension $R'/R$.
Acting as above, we construct, using~\cite[Exp. XXVI, Prop. 6.1]{SGA}, first a reductive subgroup of
$G$ of type $\Theta_\alpha$, and then consider its semisimple derived subgroup and an absolutely almost simple subgroup
$G_{\Theta_\alpha}$
of type $\Theta_\alpha$ therein. Note that the module $V_\alpha$ being 1-dimensional is equivalent to the fact
that the only simple root of the root system $\Phi$ projected onto $\alpha$ is not adjacent on the Dynkin diagram of $G$ to
any simple root of the Levi subgroup $L$.
Then the subgroup $G_\alpha'$ constructed above has type $A_1$ and root system $\{\pm\alpha\}$. Clearly,
it is contained in $G_{\Theta_\alpha}$.

If $\Psi=BC_n$, we can visualize the type of $G_\alpha=G_\ha$ if we draw the extended absolute Dynkin diagram of $G$, and
throw away all vertices corresponding to simple roots not belonging to $L$.
In this case $G_\ha$ may not be a Levi subgroup of any parabolic subgroup. However, one readily sees that
in the groups of classical type, the connected component of the Dynkin diagram $D$
of $G_\ha$, containing $-\ha$, is the type of an absolutely almost simple factor of a Levi subgroup of a parabolic subgroup of $G$ after
a splitting base extension; hence it is the type of a simply connected absolutely almost simple group.
See the list of possible Tits indices in~\cite{PS-tind}. In the exceptional groups,
this connected component always consists of one vertex, hence $X_{\pm\ha}$ are 1-dimensional, and hence contained
in a group of type $\SL_2$ after a splitting base extension. We choose the respective subgroup to be
$G_\alpha'$.
\end{proof}

{\bf Example.} Assume that $G$ over $A$ is of type ${}^1D_{12}$, and the parabolic subgroup $P$ is of type $\{4,8\}$,
which is represented by circled vertices on the Dynkin diagram below; we use standard Bourbaki numbering.
Then $\Psi=BC_2$, and $G_{\ha}$ is a reductive group
of type $D_4+A_3+D_4$. The group $G_\ha'$ is its absolutely almost simple subgroup of type $D_4$,
containing $X_{\pm\ha}$. After a splitting base extension, $G_\ha'$ is contained in the Levi subgroup of the parabolic subgroup
of $G$ of type $\{4,8\}$, but different from $P$; it is standard with respect to the basis of simple roots
obtained by adding $\ha$ to the original one, and removing, say, the 12th root.
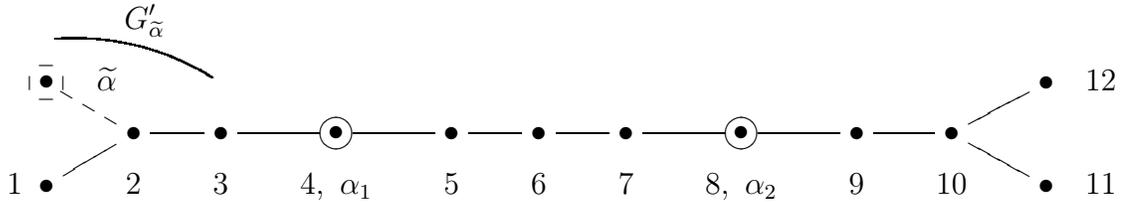
\begin{equation*}
\xymatrix@R=5pt@C=20pt{
\ar@/^/@{{}{-}{}}[rrd]^{\displaystyle G_{\ha}'}&&&&&&&&&&\\
*+[o][F--]{\bullet}\ar@{--}[rd]&\hspace{-20pt}\ha&&&&&&&&&\bullet\hbox to 0pt{\quad 12}\\
&\bullet\ar@{-}[r]&\bullet\ar@{-}[r]&*+[o][F]{\bullet}\ar@{-}[r]&\bullet\ar@{-}[r]&\bullet\ar@{-}[r]&
\bullet\ar@{-}[r]&*+[o][F]{\bullet}\ar@{-}[r]&\bullet
\ar@{-}[r]&\bullet\ar@{-}[ru]\ar@{-}[rd]\\
\bullet\ar@{-}[ru]&\hspace{-45pt}\hbox to 0pt{1}\hspace{45pt}2&3&4,\ \alpha_1&5&6&7&8,\ \alpha_2&9&10&\bullet\hbox to 0pt{\quad 11}\\
&&&&&&&&&&
}
\end{equation*}

Return to our general setting. Let $\alpha\in\Psi$ be a relative root. Lemma~\ref{lem:G_alpha} tells, in particular, that $P^\pm\cap G'_\alpha$ are the minimal
parabolic subgroups of $G'_\alpha$. Clearly, the respective relative root system of $G'_\alpha$
identifies with the subset $\{\pm\alpha,\pm 2\alpha,\ldots,\pm m_\alpha\alpha\}$ of $\Psi$.

\begin{defn}
Denote by $\tau_\alpha$ the group automorphism of
$G'_\alpha(A[X,X^{-1}])$, such that $\tau_\alpha|_{L\cap G'_\alpha}=\id_{L\cap G'_\alpha}$, and
$$
\tau_\alpha(X_{k\alpha}(u))=X_{k\alpha}(X^ku)
$$
for any $k\ge 1$ and $u\in V_{k\alpha}\otimes_A A[X,X^{-1}]$.
\end{defn}
The existence of $\tau_\alpha$ is guaranteed by Lemma~\ref{lem:sigma-def}.
Note that if $m_1(\alpha)=k$, then $\sigma|_{E_\alpha(A[X,X^{-1}])}=(\tau_\alpha)^k$.

\begin{lem}\label{lem:Abe-tauX}
Let $\alpha\in\Psi'$ be a non-multipliable relative root. Then we have
$$
\tau_\alpha^{\pm 1}\l(E_\alpha(A[X],XA[X])\r)\subseteq G'_\alpha(A[X])\cap E(A[X]).
$$
\end{lem}
\begin{proof}
We consider the automorphism $\tau_\alpha$, the case of $\tau_\alpha^{-1}$ is symmetric.
By Lemma~\ref{lem:Abe-sect}, any $x\in E_\alpha(A[X],XA[X])$ is a product of $Z_{\pm\alpha}(a,Xf)$, where $a\in E_\alpha(A)$ and $f\in V_{\pm\alpha}\otimes_A A[X]$.
By Lemma~\ref{lem:n_alpha} there is $n_\alpha\in E_\alpha(A)$ such that $n_\alpha U_{(\alpha)}n_\alpha^{-1}\subseteq U_{(-\alpha)}$
and vice versa.
Hence
$$
Z_{-\alpha}(a,Xf)=an_\alpha^{-1}(n_\alpha X_{-\alpha}(Xf)n_\alpha^{-1})n_\alpha a^{-1}=an_\alpha^{-1}X_{\alpha}(Xf')n_\alpha
a^{-1}=
Z_\alpha(an_\alpha^{-1},Xf'),
$$
for some $f'\in V_\alpha\otimes_A A[X]$.
Therefore, we only need to check that $\tau_\alpha (Z_{\alpha}(a,Xf))$ belongs to $E(A[X])$ for any
$a\in E_\alpha(A)$, $f\in V_\alpha\otimes_A A[X]$.

By Gauss decomposition in $G'_\alpha(A)$~\cite[Exp. XXVI, Th\'eor\`eme 5.1]{SGA} we have
$$
a=lX_\alpha(a_1)X_{-\alpha}(b)X_\alpha(a_2),
$$
$a_1,a_2,b\in A$, $l\in L_\alpha(A)$. Then
$\tau_\alpha(a)=lX_\alpha(a_1X)X_{-\alpha}(bX^{-1})X_\alpha(a_2X)$. Then
$$
\tau_\alpha\l(Z_{\alpha}(a,Xf)\r)=\tau_\alpha\l({}^aX_\alpha(Xf)\r)={}^{\tau_\alpha(a)}X_\alpha(X^2f)
$$
belongs to $E(A[X])$ if and only if
$$
X_{-\alpha}(bX^{-1})X_\alpha(a_2X)X_\alpha(X^2f)\l(X_{-\alpha}(bX^{-1})X_\alpha(a_2X)\r)^{-1}=
{}^{X_{-\alpha}(bX^{-1})}X_\alpha(X^2f)
$$
does. In the rest of the proof, we show that
\begin{equation}\label{eq:tauX-main}
{}^{X_{-\alpha}(bX^{-1})}X_\alpha(X^2f)\in E(A[X]).
\end{equation}
Note that $\alpha$ belongs to a root subsystem of $\Psi$ of type $A_2$, $B_2$, or is a short root in $G_2$.

{\bf Case 1.} Assume that $\alpha$ belongs to a root
subsystem of type $A_2$. Then
$$
X_{\alpha}(bX^2)=[X_\beta(uX),X_\gamma(vX)],
$$
where $u\in V_\beta$,
$v\in V_\gamma$, $\beta+\gamma=\alpha$, $\beta,\gamma$ non-collinear to $\alpha$, by~\cite[Lemma 2]{LS}. Then by the generalized
Chevalley commutator formula,
$$
\begin{array}{rcl}
{}^{X_{-\alpha}(bX^{-1})}\l(X_\beta(uX)^{\pm 1}\r)&=&[X_{-\alpha}(bX^{-1}),X_\beta(\pm uX)]X_\beta(\pm uX)\\
&=&X_{-\gamma}\bigl(N_{-\alpha,\beta,1,1}(bX^{-1},\pm uX)\bigr)X_\beta(\pm uX)\\
&=&X_{-\gamma}\bigl(N_{-\alpha,\beta,1,1}(b,\pm u)\bigr)X_\beta(\pm uX)
\end{array}
$$
belongs to $E(A[X])$. Similarly, ${}^{X_{-\alpha}(bX^{-1})}\l(X_\gamma(vX)^{\pm 1}\r)$ belongs
to $E(A[X])$. Therefore, ${}^{X_{-\alpha}(bX^{-1})}X_\alpha(X^2f)$ belongs to $E(A[X])$.

{\bf Case 2.} Assume that $\alpha$ is a long root in a subsystem of type $B_2$. Let $\beta$ be a short root such that $\alpha,\beta$ is a system of simple roots
for $B_2$. By~\cite[Lemma 2 (2)]{LS} there are $u\in V_{\alpha+\beta}\otimes_A A[X]$, $v\in V_{-\beta}\otimes_A A[X]$,
and $c_i\in V_{\alpha+2\beta}\otimes_A A[X]$, $d_i\in V_{-\beta}\otimes_A A[X]$, $1\le i\le k$, such that
$$
f=N_{\alpha+\beta,-\beta,1,1}(u,v)+\sum\limits_{i=1}^kN_{\alpha+2\beta,-\beta,1,2}(c_i,d_i).
$$
By the generalized Chevalley commutator formula, this means that $X_\alpha(X^2f)$ equals
$$
\l[X_{\alpha+\beta}(Xu),X_{-\beta}(Xv)\r]\prod\limits_{i=1}^k\Bigl([X_{\alpha+2\beta}(c_i),X_{-\beta}(Xd_i)]
X_{\alpha+\beta}\bigl(-XN_{\alpha+2\beta,-\beta,1,1}(c_i,d_i)\bigr)\Bigr).
$$
We conjugate this expression by $X_{-\alpha}(bX^{-1})$.
Since a long root in $B_2$
cannot be added to another root more than once, by the generalized Chevalley commutator formula,
the $X_{-\alpha}(bX^{-1})$-conjugates of the factors $X_{\alpha+\beta}(Xu)$, $X_{-\beta}(Xv)$, $X_{-\beta}(Xd_i)$,
and $X_{\alpha+\beta}\bigl(-XN_{\alpha+2\beta,-\beta,1,1}(c_i,d_i)\bigr)$ belong to $E(A[X])$.
Since $(-\alpha)+(\alpha+2\beta)$ is not a root, the $X_{-\alpha}(bX^{-1})$-conjugate of $X_{\alpha+2\beta}(c_i)$
equals $X_{\alpha+2\beta}(c_i)$, and hence also belongs to $E(A[X])$. Consequently,
we again have ${}^{X_{-\alpha}(bX^{-1})}X_\alpha(X^2f)\in E(A[X])$.

{\bf Case 3.} Assume that $\alpha$ is a short root in a subsystem of type $B_2$. Let $\beta$ be a long root in this $B_2$ such that
$\alpha,\beta$ form a system of simple roots. By~\cite[Lemma 2 (1b)]{LS}, since $(-\beta)-(\alpha+\beta)$ is not a root,
we can write
$$
X_{\alpha}(X^2f)=[X_{-\beta}(uX),X_{\alpha+\beta}(vX)]X_{2\alpha+\beta}(wX^3),
$$
 for some $u\in V_{-\beta}$,
$v\in V_{\alpha+\beta}$, $w\in V_{2\alpha+\beta}$. By the generalized Chevalley commutator formula,
${}^{X_{-\alpha}(bX^{-1})}X_{2\alpha+\beta}(wX^3)\in E(A[X])$. On the other hand,
\begin{equation}\label{eq:tauX-B2}
\begin{array}{l}
{}^{X_{-\alpha}(bX^{-1})}[X_{-\beta}(uX),X_{\alpha+\beta}(vX)]=
\bigl[{}^{X_{-\alpha}(bX^{-1})}X_{-\beta}(uX),\,{}^{X_{-\alpha}(bX^{-1})}X_{\alpha+\beta}(vX)\bigr]\\
=\bigl[X_{-\alpha-\beta}(c_1)X_{-2\alpha-\beta}(c_2X^{-1})X_{-\beta}(uX),\,
X_\beta(c_3)X_{\alpha+\beta}(vX)\bigr],
\end{array}
\end{equation}
for some $c_1\in V_{-\alpha-\beta}$, $c_2\in V_{-2\alpha-\beta}$, $c_3\in V_\beta$. Note that $X_{-2\alpha-\beta}(c_2X^{-1})$
commutes with all other root factors involved in the last expression of~\eqref{eq:tauX-B2}, except for $X_{\alpha+\beta}(vX)$, and the commutator
with the latter is equal to
$$
[X_{-2\alpha-\beta}(c_2X^{-1}),X_{\alpha+\beta}(vX)]=X_{-\alpha}(c_4)X_\beta(c_5X),
$$
for some $c_4\in V_{-\alpha}$, $c_5\in V_\beta$. Thus, we can safely cancel out the factor
$X_{-2\alpha-\beta}(c_2X^{-1})$ of~\eqref{eq:tauX-B2}, which is the only one involving $X^{-1}$, with its inverse. Therefore,
$$
{}^{X_{-\alpha}(bX^{-1})}X_{\alpha}(X^2f)={}^{X_{-\alpha}(bX^{-1})}[X_{-\beta}(uX),X_{\alpha+\beta}(vX)]
{}^{X_{-\alpha}(bX^{-1})}X_{2\alpha+\beta}(wX^3).
$$
belongs to $E(A[X])$.

{\bf Case 4.} Assume that $\alpha$ is a short root in a subsystem of type $G_2$. Let $\beta$ denote a long root in this $G_2$ such that
$\alpha,\beta$ form a system of simple roots. Since $(\alpha+\beta)-(-\beta)$ is not a root,
by Lemma~\cite[Lemma 2 (1b)]{LS} and the generalized Chevalley commutator formula, we can write
$$
X_\alpha(X^2f)=[X_{\alpha+\beta}(u),X_{-\beta}(X^2v)]X_{2\alpha+\beta}(X^2w_1)X_{3\alpha+2\beta}(X^2w_2)
X_{3\alpha+\beta}(X^4w_3),
$$
for some $u\in V_{\alpha+\beta}\otimes_A A[X]$, $v\in V_{-\beta}\otimes_A A[X]$, etc. One readily sees that
by the generalized Chevalley commutator formula ${}^{X_{-\alpha}(X^{-1}b)}X_{3\alpha+2\beta}(X^2w_2)
=X_{3\alpha+2\beta}(X^2w_2)$, as well as  ${}^{X_{-\alpha}(X^{-1}b)}X_{2\alpha+\beta}(X^2w_1)$ and
${}^{X_{-\alpha}(X^{-1}b)}X_{3\alpha+\beta}(X^4w_3)$, all belong to $E(A[X])$. On the other hand, we have
\begin{multline}\label{eq:G2}
{}^{X_{-\alpha}(X^{-1}b)}\l[X_{\alpha+\beta}(u),X_{-\beta}(X^2v)\r]=\\
=\bigl[X_{\alpha+\beta}(u)X_\beta(X^{-1}c_1),X_{-\beta}(X^2v)X_{-\alpha-\beta}(Xc_2)X_{-2\alpha-\beta}(c_3)\cdot\\
\cdot X_{-3\alpha-\beta}(X^{-1}c_4)X_{-3\alpha-2\beta}(Xc_5)\bigr],
\end{multline}
where $c_1\in V_\beta\otimes_A A[X]$ etc. Note that $X_{-3\alpha-\beta}(X^{-1}c_4)$ commutes with
 $X_{\alpha+\beta}(- u)$ and $X_{\beta}(-X^{-1}c_1)$, since the sums of respective roots are not roots;
hence we can cancel out the factor $X_{-3\alpha-\beta}(X^{-1}c_4)$ with its inverse in~\eqref{eq:G2} without
modifying anything else. Then we can also eliminate $X_\beta(X^{-1}c_1)$. Indeed, by the $A_1$ case considered above,
we have ${}^{X_\beta(X^{-1}c_1)}X_{-\beta}(X^2v)\in E(A[X])$; and by the generalized Chevalley commutator formula,
$X_\beta(X^{-1}c_1)$ commutes with $X_{-2\alpha-\beta}(c_3)$, and its commutators with $X_{-\alpha-\beta}(Xc_2)$
and $X_{-3\alpha-2\beta}(Xc_5)$ belong to $E(A[X])$. This implies that
the expression~\eqref{eq:G2} belongs to $E(A[X])$, and hence ${}^{X_{-\alpha}(X^{-1}b)}X_\alpha(X^2f)$
belongs to $E(A[X])$.
\end{proof}

Observe that Lemma~\ref{lem:Abe-tauX} does not cover the case where $\Psi$ is of type $BC_l$ and $\alpha$ is an extra-short root. To treat this
case, we prove two preparatory lemmas. The following digram reflects the position of simple roots in $BC_2$.

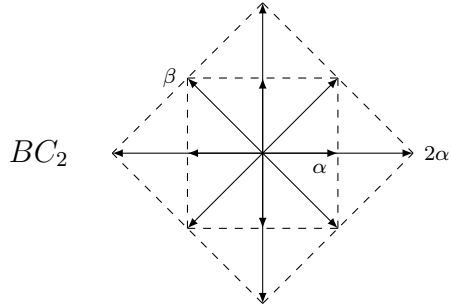
\vspace{5pt}
\begin{center}
\begin{tikzpicture}[>=latex]

\pgfmathsetmacro\ax{2*cos(45)}
\draw[->] (0,0) -- (0,1);
\draw[->] (0,0) -- (0,2);
\draw[->] (0,0) -- (0,-1);
\draw[->] (0,0) -- (0,-2);
\draw[->] (0,0) -- (1,0) node[below left] {$\scriptstyle\alpha$};
\draw[->] (0,0) -- (2,0) node[right] {$\scriptstyle2\alpha$};
\draw[->] (0,0) -- (-1,0);
\draw[->] (0,0) -- (-2,0)  node[left] {\hbox to 35pt{$BC_2$}};
\draw[->] (0,0) -- (1,1);
\draw[->] (0,0) -- (-1,1) node[left] {$\scriptstyle\beta$};
\draw[->] (0,0) -- (1,-1) ;
\draw[->] (0,0) -- (-1,-1);
\draw[dashed] (0,2) -- (2,0);
\draw[dashed] (0,2) -- (-2,0);
\draw[dashed] (0,-2) -- (2,0);
\draw[dashed] (0,-2) -- (-2,0);
\draw[dashed] (1,1) -- (1,-1);
\draw[dashed] (1,1) -- (-1,1);
\draw[dashed] (-1,1) -- (-1,-1);
\draw[dashed] (1,-1) -- (-1,-1);
\end{tikzpicture}
\end{center}

\begin{lem}\label{lem:BC-1}
Assume that $\Psi=BC_l$, $l\ge 2$. Let $\alpha,\beta\in\Psi$ be two simple roots in a root subsystem of
type $BC_2$, with $\alpha$ extra-short.
Consider the subgroup $Y_{\alpha,\beta}^+$ of $E(A[X])$ given by
$$
\begin{array}{rcl}
Y_{\alpha,\beta}^+&=&\bigl<X_{-\beta}(V_{-\beta}\otimes_AXA[X]),\ X_{\pm(\alpha+\beta)}(V_{\pm(\alpha+\beta)}\otimes_A A[X]),\\
&&\ X_{\pm2(\alpha+\beta)}(V_{\pm2(\alpha+\beta)}\otimes_A A[X]),\ X_\beta(V_\beta\otimes_A A[X]),\\
&&\ X_{-2\alpha-\beta}(V_{-2\alpha-\beta}\otimes_A A[X]),\ X_{2\alpha+\beta}(V_{2\alpha+\beta}\otimes_A XA[X])\bigr>.
\end{array}
$$
Then

(i) $Y_{\alpha,\beta}^+$ contains $X_{2\alpha}(V_{2\alpha}\otimes_A X^2A[X])$, $X_{-2\alpha}(V_{-2\alpha}\otimes_A A[X])$,
$X_\alpha(V_\alpha\otimes_AXA[X])$, and $X_{-\alpha}(V_{-\alpha}\otimes_A A[X])$.

(ii) $Y_{\alpha,\beta}^+$ is normalized by
$X_{-2\alpha}(V_{-2\alpha}\otimes_AX^{-1}A[X])$.
\end{lem}
\begin{proof}
(i) By the generalized Chevalley commutator formula and~\cite[Lemma 2 (2)]{LS}, for any $u\in V_{2\alpha}\otimes_A A[X]$,
there are $v\in V_{2\alpha+\beta}\otimes_A A[X]$, $w\in V_{-\beta}\otimes_A A[X]$, and
$c_1,\ldots,c_m\in V_{2(\alpha+\beta)}\otimes_A A[X]$, $d_1,\ldots,d_m\in V_{-\beta}\otimes_A A[X]$,
such that $X_{2\alpha}(X^2u)$ equals
\begin{equation*}
\begin{array}{c}
\bigl[X_{2\alpha+\beta}(Xv),X_{-\beta}(Xw)\bigr]\prod\limits_{i=1}^m\Bigl(\bigl[X_{2(\alpha+\beta)}(c_i),X_{-\beta}(Xd_i)\bigr]
X_{2\alpha+\beta}(-XN_{2(\alpha+\beta),-\beta,1,1}(c_i,d_i))\Bigr).
\end{array}
\end{equation*}
Hence $X_{2\alpha}(V_{2\alpha}\otimes_A X^2A[X])\subseteq Y_{\alpha,\beta}^+$. On the other hand, by~\cite[Lemma 2 (1), case (b)]{LS},
for any $u\in V_\alpha\otimes_A A[X]$ there are $v\in V_{\alpha+\beta}\otimes_A A[X]$, $w\in V_{-\beta}\otimes_A A[X]$,
such that $u=N_{\alpha+\beta,-\beta,1,1}(v,w))$, i.e.
\begin{equation}\label{eq:BC2}
\begin{array}{rcl}
X_\alpha(Xu)&= \bigl[X_{\alpha+\beta}(v), X_{-\beta}(Xw)\bigr] X_{2\alpha+\beta}(-XN_{\alpha+\beta,-\beta,2,1}(v,w))\cdot\\
&\qquad\cdot   X_{2\alpha}(-X^2N_{\alpha+\beta,-\beta,1,2}(v,w)).
\end{array}
\end{equation}
Hence $X_\alpha(V_\alpha\otimes_AXA[X])\subseteq Y_{\alpha,\beta}^+$.

The cases of $X_{-2\alpha}(V_{-2\alpha}\otimes_A A[X])$ and $X_{-\alpha}(V_{-\alpha}\otimes_A A[X])$ are treated
in the same way, using the opposite roots.

The statement (ii) follows from the generalized Chevalley commutator formula.
\end{proof}

\begin{lem}\label{lem:BC-2}
Assume the hypothesis of Lemma~\ref{lem:BC-1}. For any
$$
g\in X_{-2\alpha}\l(V_{-2\alpha}\otimes_AX^{-2}A[X]\r)X_{-\alpha}\l(V_{-\alpha}\otimes_AX^{-1}A[X]\r),
$$
the sets

(i) $gX_{2\alpha+\beta}\l(V_{2\alpha+\beta}\otimes_A X^2A[X]\r)g^{-1}$,

(ii) $gX_{2\alpha}\l(V_{2\alpha}\otimes_A X^3A[X]\r)g^{-1}$,

(iii) $gX_{\alpha}\l(V_\alpha\otimes_A X^2A[X]\r)g^{-1}$

\noindent are all contained in $Y_{\alpha,\beta}^+\cdot X_{-2\alpha}\l(V_{-2\alpha}\otimes_AX^{-1}A[X]\r)$.
\end{lem}
\begin{proof}
Set $g=g_2g_1$, where $g_1\in X_{-\alpha}(V_{-\alpha}\otimes_AX^{-1}A[X])$, $g_2\in X_{-2\alpha}(V_{-2\alpha}\otimes_AX^{-2}A[X])$.

(i) Using the generalized Chevalley commutator formula, we deduce that
the $g_1$-conjugate of $X_{2\alpha+\beta}(V_{2\alpha+\beta}\otimes_A X^2A[X])$ is contained in
the product
$$
\begin{array}{l}
X_\beta(V_\beta\otimes_A A[X])X_{\alpha+\beta}(V_{\alpha+\beta}\otimes_A XA[X])X_{2(\alpha+\beta)}(V_{2(\alpha+\beta)}\otimes_A X^2A[X])\cdot\\
\qquad\qquad\cdot
X_{2\alpha+\beta}(V_{2\alpha+\beta}\otimes_A X^2A[X]),
\end{array}
$$
and that the $g_2$-conjugate of this product is contained in $Y_{\alpha,\beta}^+$.

(ii) To compute $gX_{2\alpha}(V_{2\alpha}\otimes_A X^3A[X])g^{-1}$, we use~\cite[Lemma 2 (2)]{LS} to conclude
that for any $u\in V_{2\alpha}\otimes_A A[X]$, there are $w\in V_{-\beta}\otimes_A A[X]$,
$v\in V_{2\alpha+\beta}\otimes_A A[X]$, and $w_i\in V_{-\beta}\otimes_A A[X]$, $v_i\in V_{2\alpha+2\beta}\otimes_A A[X]$,
$1\le i\le k$, such that
$$
u=N_{-\beta,2\alpha+\beta,1,1}(w,v)+\sum\limits_{i=1}^kN_{-\beta,2\alpha+2\beta,2,1}(w_i,v_i).
$$
Therefore, one has
\begin{equation}\label{eq:2aX3-pres}
\begin{array}{rl}
X_{2\alpha}(X^3u)&=\bigl[X_{-\beta}(Xw),X_{2\alpha+\beta}(X^2v)\bigr]\cdot\\
&\qquad\cdot\prod\limits_{i=1}^k\Bigl(\bigl[X_{-\beta}(Xw_i),X_{2\alpha+2\beta}(Xv_i)\bigr]
X_{2\alpha+\beta}(X^2s_i)
\Bigr),
\end{array}
\end{equation}
where $s_i=-N_{-\beta,2\alpha+2\beta,1,1}(w_i,v_i)$, $1\le i\le k$. First we treat the first commutator in~\eqref{eq:2aX3-pres}.
We have
\begin{equation}\label{eq:2aX3}
\begin{array}{rl}\
{}^{g_1}\l[X_{-\beta}(Xw),X_{2\alpha+\beta}(X^2v)\r]
&=\bigl[X_{-\beta}(Xw)X_{-\alpha-\beta}(c_1)X_{-2(\alpha+\beta)}(c_2)X_{-2\alpha-\beta}(X^{-1}c_3),\\
&\qquad X_{2\alpha+\beta}(X^2v)X_\beta(c_4)X_{\alpha+\beta}(Xc_5)X_{2(\alpha+\beta)}(X^2c_6)
\bigr],
\end{array}
\end{equation}
where $c_1\in V_{-\alpha-\beta}\otimes_A A[X]$, $c_2\in V_{-2(\alpha+\beta)}\otimes_A A[X]$, etc. One readily
sees that conjugating this expression by $g_2$ does not change its shape, so we can skip this operation.
Now, using~\cite[Lemma 2 (2)]{LS} again, we can express the factor $X_{2\alpha+\beta}(X^2v)$ in the right-hand side
of~\eqref{eq:2aX3} as a product
$$
\begin{array}{l}
\bigl[X_{\alpha}(Xv_1),X_{\alpha+\beta}(Xv_2)\bigr] \cdot \bigl[X_{2\alpha}(X^2v_3),X_\beta(v_4)\bigr]\cdot
X_{2\alpha+2\beta}\l(-X^2N_{2\alpha,\beta,1,2}(v_3,v_4)\r)\cdot\\
\qquad \cdot
\prod\limits_{i=1}^m\Bigl(\bigl[X_\alpha(Xc_i),X_\beta(d_i)\bigr] X_{\alpha+\beta}\l(-XN_{\alpha,\beta,1,1}(c_i,d_i)\r)
X_{2(\alpha+\beta)}\l(-X^2N_{\alpha,\beta,2,2}(c_i,d_i)\r)\Bigr)
,
\end{array}
$$
for some elements $v_1,c_1,\ldots,c_m\in V_{\alpha}\otimes_A A[X]$, $v_2\in V_{\alpha+\beta}\otimes_A A[X]$,
$v_3\in V_{2\alpha}\otimes_A A[X]$, $v_4,d_1,\ldots,d_m\in V_\beta\otimes_A A[X]$. Now we observe that
the commutator of $X_{-2\alpha-\beta}(X^{-1}c_3)$ with an element
of any of the sets
$$
\begin{array}{l}
X_{\alpha}(V_\alpha\otimes_A XA[X])^{\pm 1},\ X_{\alpha+\beta}(V_{\alpha+\beta}\otimes_A XA[X])^{\pm 1},\
X_\beta(V_\beta\otimes_A A[X]),\\
X_{2\alpha}(V_{2\alpha}\otimes_A X^2A[X]),\  X_{2(\alpha+\beta)}(V_{2(\alpha+\beta)}\otimes_A X^2A[X])
\end{array}
$$
belongs to the product
$Y_{\alpha,\beta}^+\cdot X_{-2\alpha}(V_{-2\alpha}\otimes_AX^{-1}A[X])$. Therefore,
since $Y_{\alpha,\beta}^+$ is normalized by
$X_{-2\alpha}(V_{-2\alpha}\otimes_AX^{-1}A[X])$ by Lemma~\ref{lem:BC-1} (ii),
we can cancel
the factor $X_{-2\alpha-\beta}(X^{-1}c_3)$ in~\eqref{eq:2aX3} with its inverse, so that the resulting expression belongs to
$Y_{\alpha,\beta}^+\cdot X_{-2\alpha}(V_{-2\alpha}\otimes_AX^{-1}A[X])$.
Consequently,  we have
$$
{}^g\l[X_{-\beta}(Xw),X_{2\alpha+\beta}(X^2v)\r]\in Y_{\alpha,\beta}^+\cdot X_{-2\alpha}(V_{-2\alpha}\otimes_AX^{-1}A[X]).
$$

Another type of factors occuring in~\eqref{eq:2aX3-pres} are factors $X_{2\alpha+\beta}(X^2s_i)$, $1\le i\le k$.
By the case (i) we have $gX_{2\alpha+\beta}(X^2s_i)g^{-1}\in Y_{\alpha,\beta}^+$.

Finally, consider ${}^g\l[X_{-\beta}(Xw_i),X_{2\alpha+2\beta}(Xv_i)\r]$, $1\le i\le k$. By the generalized Chevalley
commutator formula, exactly as above we have
$$
\begin{array}{rl}
{}^g\bigl[X_{-\beta}(Xw_i),X_{2\alpha+2\beta}(Xv_i)\bigr]&=
\bigl[X_{-\beta}(Xw_i)X_{-\alpha-\beta}(c_{1i})X_{-2(\alpha+\beta)}(c_{2i})X_{-2\alpha-\beta}(X^{-1}c_{3i}),\\
&\qquad X_{2\alpha+2\beta}(Xv_i)\bigr]\\
\end{array}
$$
for some $c_{1i}\in V_{-\alpha-\beta}\otimes_A A[X]$, $c_{2i}\in V_{-2(\alpha+\beta)}\otimes_A A[X]$, etc.
Note that
$$
\begin{array}{rl}
[X_{-2\alpha-\beta}(X^{-1}c_{3i}),X_{2\alpha+2\beta}(Xv_i)]&=
X_{\beta}\l(N_{-2\alpha-\beta,2\alpha+2\beta,1,1}(c_{3i},v_i)\r)\cdot\\
&\qquad\cdot X_{-2\alpha}\l(X^{-1}N_{-2\alpha-\beta,2\alpha+2\beta,2,1}(c_{3i},v_i)\r).
\end{array}
$$
Hence, by Lemma~\ref{lem:BC-1} we conclude that
$$
g[X_{-\beta}(Xw_i),X_{2\alpha+2\beta}(Xv_i)]g^{-1}\in Y_{\alpha,\beta}^+\cdot X_{-2\alpha}(V_{-2\alpha}\otimes_A X^{-1}A[X]).
$$

Applying all these results to the expression in~\eqref{eq:2aX3-pres}, we deduce that
$$
gX_{\alpha}(V_\alpha\otimes_A X^2A[X])g^{-1}\subseteq Y_{\alpha,\beta}^+\cdot X_{-2\alpha}(V_{-2\alpha}\otimes_A X^{-1}A[X]).
$$

(iii) To compute $gX_{\alpha}(V_\alpha\otimes_A X^2A[X])g^{-1}$, we use the same relations as in~\eqref{eq:BC2} to decompose
$X_{\alpha}(V_\alpha\otimes_A X^2A[X])$:
$$
\begin{array}{rl}
X_\alpha(V_\alpha\otimes_A X^2A[X])&\subseteq \bigl[X_{\alpha+\beta}(V_{\alpha+\beta}\otimes_A XA[X]),
  X_{-\beta}(V_{-\beta}\otimes_A XA[X])\bigr]\cdot \\
&\qquad\cdot X_{2\alpha+\beta}(V_{2\alpha+\beta}\otimes_A X^3A[X])\cdot
  X_{2\alpha}(V_{2\alpha}\otimes_A X^4A[X]).\\
\end{array}
$$
By the previous results, we only need to consider $g$-conjugates of the commutator. One readily sees that
$$
{}^{g_1}[X_{\alpha+\beta}(V_{\alpha+\beta}\otimes_A\! XA[X]),
  X_{-\beta}(V_{-\beta}\otimes_A\! XA[X])]
$$
is contained in the commutator
\begin{equation}\label{eq:-a+a}
\begin{array}{l}
 \Bigl[X_{\alpha+\beta}(V_{\alpha+\beta}\otimes_A XA[X])X_\beta(V_\beta\otimes_A A[X]),\\
\qquad X_{-\beta}(V_{-\beta}\otimes_A XA[X])
X_{-\alpha-\beta}(V_{-\alpha-\beta}\otimes_A\! A[X])X_{-2(\alpha+\beta)}(V_{-2(\alpha+\beta)}\otimes_A\! A[X])\cdot\\
\qquad\qquad
\cdot X_{-2\alpha-\beta}(V_{-2\alpha-\beta}\otimes_A\! X^{-1}A[X])\Bigr].
\end{array}
\end{equation}
Again, conjugating this expression by $g_2$ does not change its shape, so it remains to note that
$$
\l[X_{\alpha+\beta}(V_{\alpha+\beta}\otimes_A XA[X])X_\beta(V_\beta\otimes_A A[X]),
X_{-2\alpha-\beta}(V_{-2\alpha-\beta}\otimes_A X^{-1}A[X])\r]
$$
is contained in the group generated by
$$
\begin{array}{l}
X_{\alpha+\beta}(V_{\alpha+\beta}\otimes_A XA[X]),\ X_\beta(V_\beta\otimes_A A[X]),\
X_{-\alpha}(V_{-\alpha}\otimes_A A[X]),\\
X_{-2\alpha}(V_{-2\alpha}\otimes_A X^{-1}A[X]),
\end{array}
$$
and hence is contained in $Y_{\alpha,\beta}^+\cdot X_{-2\alpha}(V_{-2\alpha}\otimes_A X^{-1}A[X])$.
Then~\eqref{eq:-a+a} is contained in
$Y_{\alpha,\beta}^+\cdot X_{-2\alpha}(V_{-2\alpha}\otimes_A X^{-1}A[X])$ as well.
\end{proof}

\begin{lem}\label{lem:BC-tauX}
Assume that $\Psi=BC_l$, $l\ge 2$. Let $\alpha\in\Psi$ be an extra-short root.
Then
$$
\tau_\alpha^{\pm 1}(E_\alpha(A[X],XA[X]))\subseteq (G'_\alpha(A[X])\cap E(A[X]))\cdot
X_{\mp 2\alpha}(X^{-1}V_{\mp 2\alpha}).
$$

More precisely, if $\beta\in\Psi$ is another root such that $\alpha,\beta$ form a system of simple roots for a subsystem
of $\Psi$ of type $BC_2$, we have
$$
\tau_\alpha^{\pm 1}(E_\alpha(A[X],XA[X]))\subseteq (G'_\alpha(A[X])\cap Y_{\pm\alpha,\pm\beta}^+)\cdot
X_{\mp 2\alpha}(X^{-1}V_{\mp 2\alpha}),
$$
where the subgroup $Y_{\alpha,\beta}^+$ is defined as in Lemma~\ref{lem:BC-1}.
\end{lem}
\begin{proof}
The cases of $\tau_\alpha$ and $\tau_\alpha^{-1}=\tau_{-\alpha}$ are symmetric, so it is enough to consider $\tau_\alpha$.
As in the first paragraph of the proof of Lemma~\ref{lem:Abe-tauX}, we conclude that $E_\alpha(A[X],XA[X])$ is generated by the
elements of
$$
aX_\alpha(V_\alpha\otimes_A XA[X])X_{2\alpha}(V_{2\alpha}\otimes_A XA[X])a^{-1},\quad a\in E_\alpha(A).
$$
 By Gauss decomposition in $E_\alpha(A)$, we have
$\tau_\alpha(a)=xyzh$, where
$$
\begin{array}{l}
x,z\in X_{\alpha}(V_{\alpha}\otimes_A XA[X])X_{2\alpha}(V_{2\alpha}\otimes_AX^2A[X]),\\
y\in X_{-\alpha}(V_{-\alpha}\otimes_AX^{-1}A[X])X_{-2\alpha}(V_{-2\alpha}\otimes_AX^{-2}A[X]),
\end{array}
$$
and $h\in L_\alpha(A)$. Clearly, conjugation by $zh$ preserves the group
$$
\textstyle \tau_\alpha\bigl(X_\alpha(V_\alpha\otimes_A XA[X])X_{2\alpha}(V_{2\alpha}\otimes_A XA[X])\bigr)=
X_\alpha(V_\alpha\otimes_A X^2A[X])X_{2\alpha}(V_{2\alpha}\otimes_A X^3A[X]).
$$
By Lemma~\ref{lem:BC-2} both
$yX_\alpha(V_\alpha\otimes_A X^2A[X])y^{-1}$ and
$yX_{2\alpha}(V_{2\alpha}\otimes_A X^3A[X])y^{-1}$ are contained in the group
$Y_{\alpha,\beta}^+\cdot X_{-2\alpha}(V_{-2\alpha}\otimes_AX^{-1}A[X])$.
Note that by Lemma~\ref{lem:BC-1} (ii)
$X_{-2\alpha}(V_{-2\alpha}\otimes_AX^{-1}A[X])$ normalizes $Y_{\alpha,\beta}^+$.
By Lemma~\ref{lem:BC-1} (i) $x\in Y_{\alpha,\beta}^+$,
hence
$$
x\cdot Y_{\alpha,\beta}^+\cdot X_{-2\alpha}(V_{-2\alpha}\otimes_AX^{-1}A[X])\cdot x^{-1}\subseteq Y_{\alpha,\beta}^+\cdot
X_{-2\alpha}(V_{-2\alpha}\otimes_AX^{-1}A[X]).
$$
This completes the proof, since $Y_{\alpha,\beta}^+\subseteq E(A[X])$.
\end{proof}

\subsection{The action of $\sigma$ on $E(A[X],X\cdot A[X])$.}\label{ssec:sigma}

\begin{lem}\label{lem:sigma(EX)-1}
Assume that $m_1(\ha)=1$, that is, $\Psi\neq G_2,F_4,E_8,BC_n$. Then
$$
\sigma^{\pm 1}\l(E(A[X],XA[X])\r)\subseteq U_1^{\mp}(A)EL_1(A)E(A[X],XA[X]).
$$
\end{lem}
\begin{proof}
It is enough to consider the case of $\sigma$. For any $\alpha\in\Psi$, the restriction $\sigma|_{E_{\alpha}(A[X,X^{-1}])}$
coincides with $\tau_\alpha$, $\tau_\alpha^{-1}=\tau_{-\alpha}$, or is trivial.
Hence by Lemma~\ref{lem:Abe-tauX} we have
$$
\sigma^{\pm 1}(E(A[X],XA[X])\subseteq E(A[X]).
$$
 Write $g\in \sigma\l(E(A[X],XA[X])\r)$ as $g=g_0g_1$,
$g_0=g(0)\in E(A)$, $g_1=g(0)^{-1}g\in E(A[X],XA[X])$. Using Gauss decomposition in $G(A)$,
write $g_0=u_1hvu_2$, where $u_1,u_2\in U_1^-(A)$, $v\in U_1^+(A)$, $h\in L_1(A)\cap E(A)=EL_1(A)$. One has
$$
\sigma^{-1}(U_1^-(A))\subseteq U_1^-\l(A[X],XA[X]\r)\subseteq E(A[X],XA[X]).
$$
Then we deduce
$$
\sigma^{-1}(g)=\sigma^{-1}(g_0)\sigma^{-1}(g_1)\in E(A[X],XA[X])h\sigma^{-1}(v)E(A[X]),
$$
where $\sigma^{-1}(v)\in U_1^+\l(A[X^{-1}],X^{-1}A[X^{-1}]\r)$. On the other hand, by the assumption
$g\in \sigma\l(E(A[X],XA[X])\r)$, we have
$\sigma^{-1}(g)\in E(A[X],XA[X])$. Consequently,
$$
\sigma^{-1}(v)\in U_1^+\l(A[X^{-1}],X^{-1}A[X^{-1}]\r)\cap E(A[X])=\{1\}.
$$
Therefore, $v=1$, and we are done.
\end{proof}

\begin{lem}~\label{lem:tau^2}
Assume that $\Psi=G_2$, $F_4$, or $E_8$. Then
$$
\sigma^{\pm 1}\l(E_\ha(A[X],XA[X])\r)\subseteq
E(A[X])X_{\mp\ha}(X^{-1}V_{\mp\ha}).
$$
\end{lem}
\begin{proof}
It is enough to consider the case of $\sigma$. We have $\sigma|_{E_\ha(E[X,X^{-1}])}=\tau_\ha^2$.
By Lemma~\ref{lem:G_alpha} there is a root $\alpha\in\Psi$, such that
$\ha,\alpha$ form a basis of a root subsystem $\Theta=\Theta_\ha$ of type $A_2$ in $\Psi$, and
the group scheme $G$ contains a split simply connected absolutely almost simple subgroup $G_{\Theta}$,
such that $X_{\beta}$, $\beta\in\Theta$, are root subgroups of $G_{\Theta}$. Then $\tau_\ha$
is also the restriction to $E_\ha(A[X,X^{-1}])$ of the automorphism $\sigma_\Theta$ of $G_{\Theta}$,
defined in the same way as $\sigma$ with respect to $\alpha_1=\ha$. Let $L_\Theta$ denote
the $\sigma_\Theta$-invariant Levi subgroup of the parabolic subgroup of $G_\Theta$ corresponding to $\alpha$
(the analogue of $L_1$ in $G$). Let $E_\Theta(-)$ denote the elementary
subgroup of $G_\Theta(-)$.
Applying the result of Lemma~\ref{lem:sigma(EX)-1}
to $G_{\Theta}$ instead of $G$, with both sides of the inclusion inverted, we deduce that
$$
\tau_\ha(E_\ha(A[X],XA[X]))\subseteq
E_\Theta(A[X],XA[X])\cdot
(L_\Theta(A)\cap E_\Theta(A))\cdot X_{-\ha}(A)X_{-\ha-\alpha}(A). 
$$
Note that, by the very definition,
$$
\tau_\ha\l(E_\ha(A[X],XA[X])\r)\subseteq E_\ha(A[X,X^{-1}])\subseteq G_\ha'(A[X,X^{-1}]).
$$
Recall that by Lemma~\ref{lem:G_alpha}, the group $G_\ha'$ is just the rank 1 split semisimple subgroup of $G_\Theta$
corresponding to the root subsystem $\{\pm\ha\}$.
Then, clearly, $L_\Theta X_{-\ha}X_{-\ha-\alpha}\cap G_\ha'=(L_\Theta\cap G_\ha')X_{-\ha}$.
Summing up, we have
$$
\tau_\ha\l(E_\ha(A[X],XA[X])\r)\subseteq
E_\Theta(A[X],XA[X])\cdot
(L_\Theta(A)\cap E_\Theta(A))\cdot X_{-\ha}(V_{-\ha}).
$$
Then, again by Lemma~\ref{lem:sigma(EX)-1},
$$
\begin{array}{rl}
\tau_\ha^2(E_\ha(A[X],XA[X]))&\subseteq \sigma_\Theta\Bigl(
E_\Theta(A[X],XA[X])\cdot (L_\Theta(A)\cap E_\Theta(A))\cdot X_{-\ha}(V_{-\ha})
\Bigr)\\
&\subseteq E_\Theta(A[X])X_{-\ha}(X^{-1}V_{-\ha}).
\end{array}
$$
Since $E_\Theta(A[X])\subseteq E(A[X])$, we are done.
\end{proof}

\begin{lem}\label{lem:sigma(EX)}
One has
$$
\sigma^{\pm 1}\l(E(A[X],XA[X])\r)\subseteq U_1^{\mp}(A)EL_1(A)E(A[X],XA[X]) X_{\mp\ha}(X^{-1}V_{\mp\ha}).
$$
\end{lem}
\begin{proof}
It is enough to consider the case of $\sigma$.
If $m_1(\ha)=1$, the claim follows from Lemma~\ref{lem:sigma(EX)-1}. Assume $m_1(\ha)=2$.
By Lemma~\ref{lem:a-nea},
any $x\in E(A[X],XA[X])$ can be presented as a product $x=x_1x_2$, where $x_1$ is a product of elements
of the groups $E_\beta(A[X],XA[X])$, where $\beta\in\Psi$ is non-collinear to $\ha$; and $x_2$ belongs to
$E_{\ha}(A[X],XA[X])$, or, respectively, $E_{\frac 12\ha}(A[X],XA[X])$,
if $\frac 12\ha$ is a root. Note that
$\ha$ is the only root $\alpha$ such that $m_1(\alpha)=2$ and $\sigma$
acts non-trivially on $X_{\pm\alpha}$; and, if $\Psi=BC_l$, the root $\frac 12\ha=\alpha$ is the only extra-short
root such that $\sigma$ acts non-trivially on $X_{\pm\alpha}$.
Then by Lemmas~\ref{lem:Abe-tauX},~\ref{lem:tau^2},
and~\ref{lem:BC-tauX}, we have
\begin{equation}\label{eq:sigma-prel}
\begin{array}{rl}
\sigma^{\pm 1}\l(E(A[X],XA[X])\r)&\subseteq E(A[X])X_{\mp\ha}(X^{-1}V_{\mp\ha})\\
&=E(A)E(A[X],XA[X])X_{\mp\ha}(X^{-1}V_{\mp\ha}).
\end{array}
\end{equation}
Write an element $g\in\sigma(E(A[X],XA[X]))$ as
$$
g=g_0g_1g_2,\qquad g_0\in E(A),\ g_1\in E(A[X],XA[X]),\ g_2\in X_{-\ha}(X^{-1}V_{-\ha}).
$$
By Gauss decomposition, we have $g_0=u_1hvu_2$, where $u_1,u_2\in U_1^-(A)$, $v\in U_1^+(A)$, $h\in EL_1(A)$.
We will compute $\sigma^{-1}(g)$ using this factorization. Inverting both sides of~\eqref{eq:sigma-prel}, we deduce that
$\sigma^{-1}(g_1)\in X_{\ha}(X^{-1}c)E(A[X])\ \mbox{for some}\ c\in V_{\ha}$. Clearly,
$\sigma^{-1}(g_2)\in E(A[X])$, since $m_1(-\ha)=-2$. Hence
$$
\sigma^{-1}(g_1g_2)\in X_{\ha}(X^{-1}c)E(A[X]).
$$
Write
$$
u_2=\biggl(\prod_{\begin{array}{c}
\vspace{-3pt}\st m_1(\alpha)=-1,\\
\st \alpha\neq -\frac\ha 2
\end{array}}\hspace{-15pt} X_\alpha(c_\alpha)\biggr) X_{-\frac\ha 2}(d)X_{-\ha}(e),
$$ where
$c_\alpha$, $d$, $e$ belong to the respective root modules. Here and below, the factor from $X_{-\frac\ha 2}$
should be omitted if $\Psi\neq BC_l$. Then
\begin{equation}\label{eq:u_2}
\sigma^{-1}(u_2)=\biggl(\prod_{
\begin{array}{c}
\vspace{-3pt}\st m_1(\alpha)=-1,\\
\st \alpha\neq -\frac\ha 2
\end{array}}\hspace{-15 pt}X_\alpha(Xc_\alpha)\biggr) X_{-\frac\ha 2}(Xd)X_{-\ha}(X^2e).
\end{equation}
Now, applying Lemma~\ref{lem:BC-1} if $\Psi=BC_n$, or, respectively, the inclusion~\eqref{eq:tauX-main}
from the proof of Lemma~\ref{lem:Abe-tauX} if $\Psi=G_2,F_4,E_8$,
we obtain
$$
\begin{array}{rl}
X_{-\frac\ha 2}(Xd)X_{-\ha}(X^2e)X_{\ha}(X^{-1}c)&=X_{\ha}(X^{-1}c)\cdot {\strut}^{X_{\ha}(-X^{-1}c)}
\!\l(X_{-\frac\ha 2}(Xd)X_{-\ha}(X^2e)\r)\\
&\in X_{\ha}(X^{-1}c)E(A[X]).
\end{array}
$$
By the generalized Chevalley commutator formula, since $\ha\in\Psi$ is a root of maximal length, we have
$$
\biggl(\prod_{
\begin{array}{c}
\vspace{-3pt}\st m_1(\alpha)=-1,\\
\st \alpha\neq -\frac\ha 2
\end{array}}\hspace{-15pt}X_\alpha(Xc_\alpha)\biggr)X_{\ha}(X^{-1}c)\in X_{\ha}(X^{-1}c)E(A[X]).
$$
Summing up, we have
$$
\sigma^{-1}(u_2g_1g_2)\in X_{\ha}(X^{-1}c)E(A[X]),
$$
and, consequently,
$$
\sigma^{-1}(g)=\sigma^{-1}(u_1h)\sigma^{-1}(v)\sigma^{-1}(u_2g_1g_2)\in E(A[X])\sigma^{-1}(v)X_{\ha}(X^{-1}c)E(A[X]),
$$
where, moreover, $\sigma^{-1}(v)\in U_1^+(A[X^{-1}],X^{-1}A[X^{-1}])$. Recall that, on the other hand,
$\sigma^{-1}(g)$ belongs to $E(A[X],XA[X])$ by the definition of $g$. Then we have
\begin{equation}\label{eq:v-ha}
\sigma^{-1}(v)X_{\ha}(X^{-1}c)\in E(A[X^{-1}],X^{-1}A[X^{-1}])\cap E(A[X])=\{1\}.
\end{equation}
Write
$$
v=\biggl(\hspace{2pt}\prod_{\st m_1(\alpha)=1}
\hspace{-5pt} X_\alpha(f_\alpha)\biggr) X_{\ha}(f),
$$
where $f_\alpha$, $f$ belong to the respective root modules. Then
\begin{equation}\label{eq:u_2}
\sigma^{-1}(v)=\biggl(\hspace{2pt}\prod_{
\st m_1(\alpha)=1}\hspace{-5 pt}X_\alpha(X^{-1}f_\alpha)\biggr) X_{\ha}(X^{-2}f).
\end{equation}
Then, clearly,~\eqref{eq:v-ha} implies that $\sigma^{-1}(v)=X_{\ha}(X^{-1}c)=1$.
Consequently, $v=1$ and hence
$$
g\in U_1^-(A)EL_1(A)E(A[X],XA[X])X_{-\ha}(X^{-1}V_{-\ha}).
$$
\end{proof}

Recall that $\rho:A\to A/I=l$ is the reduction modulo the maximal ideal map. We denote by the same letter the
associated maps on the groups of points of $G$. Note that Lemma~\ref{lem:sigma-def} implies that
$$
\sigma\circ\rho=\rho\circ\sigma.
$$

\begin{lem}\label{lem:sigma-M}
One has
$$
\begin{array}{c}
\sigma^{\pm 1}\bigl(E(A[X],XA[X])\cap\ker\rho\bigr)\subseteq X_{\mp\ha}(X^{-1}IV_{\mp\ha})M_+^*;\\
\sigma^{\pm 1}\bigl(E(A[X^{-1}],X^{-1}A[X^{-1}])\cap\ker\rho\bigr)\subseteq X_{\pm\ha}(XIV_{\pm\ha})M_-^*.
\end{array}
$$
Consequently, $\sigma(M_+^*M_-^*)=M_+^*M_-^*$.
\end{lem}
\begin{proof}
To prove the first inclusion, recall that for any $g\in E(A[X],XA[X])$ we have
$\sigma(g)=X_{-\ha}(X^{-1}u)h$ for some $u\in V_{-\ha}$, $h\in E(A[X])$, by Lemma~\ref{lem:sigma(EX)}. If, moreover,
$\rho(g)=1$, then $\rho(\sigma(g))=\sigma(\rho(g))=1$ as well. Hence,
$$\rho(h)^{-1}=\rho(X_{-\ha}(X^{-1}u))=X_{-\ha}(X^{-1}\rho(u))
$$
belongs
to $E(l[X])$. This implies that $\rho(u)=0$, or $u\in IV_{-\ha}$. Automatically, $h\in M_+^*$. The first inclusion is proved.
The second inclusion follows by symmetry.

Now take any $y\in M_+^*$ and $z\in M_-^*$. We can write $y=y_1y_0$, $y_0=y(0)\in E(A)$, and $y_1=yy_0^{-1}\in E(A[X],XA[X])$.
Similarly, $z=z_0z_1$, where $z_0=z(\infty)\in E(A)$, and $z_1=z(\infty)^{-1}z\in E(A[X^{-1}],X^{-1}A[X^{-1}])$.
Then $yz=y_1y_0z_0z_1$.
Clearly,
$$
\rho(y_0)=\rho(y_1)=\rho(z_0)=\rho(z_1)=1.
$$
By Lemma~\ref{lem:E*(A,I)} the product $x=y_0z_0$ belongs to $U_1^-(I)\l(L_1(A,I)\cap E(A)\r)U_1^+(I)$; write
$x=x^-tx^+$. Since $E(A)$ normalizes $E(A[X],XA[X])$ and
$E(A[X^{-1}],X^{-1}A[X^{-1}])$, we can rewrite
\begin{equation}
\label{eq:zzz}
yz=x^-ty_1z_1x^+,
\end{equation}
for some new $y_1\in E(A[X],XA[X])\cap \ker\rho$ and $z_1\in E(A[X^{-1}],X^{-1}A[X^{-1}])\cap \ker\rho$.

Applying the first part of the lemma, together with the inversion of both sides where needed, we deduce
$$
\sigma^{-1}(y_1z_1)\in M_+^*X_{\ha}(X^{-1}IV_{\mp\ha})X_{-\ha}(XIV_{\pm\ha})M_-^*.
$$
Again by Lemma~\ref{lem:E*(A,I)}, we have
$$
\begin{array}{rl}
X_{\ha}(X^{-1}IV_{\mp\ha})X_{-\ha}(XIV_{\pm\ha})&=\tau_{\ha}^{-1}\bigl(X_{\ha}(IV_{\mp\ha})X_{-\ha}(IV_{\pm\ha})\bigr)\vspace{2pt}\\
&\vspace{2pt}\subseteq\tau_{\ha}^{-1}\Bigl(X_{-\ha}(IV_{\pm\ha})\bigl(L_1(A,I)\cap E(A)\bigr)X_{\ha}(IV_{\mp\ha})\Bigr)\\
&=X_{-\ha}(IXV_{\pm\ha})\bigl(L_1(A,I)\cap E(A)\bigr)X_{\ha}(IX^{-1}V_{\mp\ha}).
\end{array}
$$
Consequently, $\sigma^{-1}(y_1z_1)\in M_+^*M_-^*$.
Now~\eqref{eq:zzz} readily implies that $\sigma^{-1}(yz)\in M_+^*M_-^*$. Thus, we have proved
$\sigma^{-1}(M_+M_-^*)\subseteq M_+^*M_-^*$.
By symmetry, $\sigma(M_+M_-^*)\subseteq M_+^*M_-^*$. Hence
$\sigma(M_+^*M_-^*)=M_+^*M_-^*$.
\end{proof}

\begin{lem}\label{lem:M_+^*}
One has $M_-^*E(A[X])\subseteq E(A[X])M_-^*$.
\end{lem}
\begin{proof}
The group $E(A[X])$ is generated by $U_1^\pm(A[X])$ by the main theorem of~\cite{PS}.
Hence any element of this group is a product
of elements of the form $X_\alpha(X^ku)$, for $\alpha\in\Psi$ such that $m_1(\alpha)\neq 0$, and $u\in V_\alpha$, $k\ge 0$.

We show by induction on $k$ that $X_\alpha(X^ku)zX_\alpha(X^ku)^{-1}\in M_+^*M_-^*$ for any $z\in M_-^*$ and any $k\ge 0$.
We can assume that
$$
z\in E(A[X^{-1}],X^{-1}A[X^{-1}])\cap \ker\rho,
$$
since $X_\alpha(X^ku)z(0)X_\alpha(X^ku)^{-1}\in E(A[X])\cap\ker\rho=M_+^*$.  By symmetry, we can also assume
$\alpha\in\Psi^+$, i.e. $m_1(\alpha)>0$.
We have either $m_1(\alpha)=1$, or $m_1(\alpha)=2$, $\alpha=\ha$.

Since $M_-^*$ is normalized by $E(A)$, the case $k=0$ is clear. Take $k>0$.
Set $y=X_\alpha(X^{k-m_1(\alpha)}u)\sigma^{-1}(z)X_\alpha(X^{k-m_1(\alpha)}u)^{-1}$,
so that
$$
X_\alpha(X^ku)zX_\alpha(X^ku)^{-1}=\sigma(y).
$$
We will show that $y\in M_+^*M_-^*$. Then the proof is finished by the second part of Lemma~\ref{lem:sigma-M},
which says that $\sigma(M_+^*M_-^*)\subseteq M_+^*M_-^*$.
Note that, given the assumptions on $z$, by the first part of Lemma~\ref{lem:sigma-M} we have
$\sigma^{-1}(z)\in X_{-\ha}(IXV_{-\ha})M_-^*$.

Assume first that $k-m_1(\alpha)\ge 0$, that is, $m_1(\alpha)=1$ or $k>1$.
Clearly, in this case
$$
X_\alpha(X^{k-m_1(\alpha)}u)X_{-\ha}(IXV_{-\ha})X_\alpha(X^{k-m_1(\alpha)}u)^{-1}\subseteq M_+^*.
$$
Since $k-m_1(\alpha)<k$, combining the above with the induction hypothesis, we deduce
$$
y\in X_\alpha(X^{k-m_1(\alpha)}u)X_{-\ha}(IXV_{-\ha})X_\alpha(X^{k-m_1(\alpha)}u)^{-1}M_+^*M_-^*\subseteq M_+^*M_-^*.
$$

Now consider the case where $m_1(\alpha)=2$ and $k=1$. Then, actually, $\alpha=\ha$ and
$X_\alpha(X^{k-m_1(\alpha)}u)=X_{\ha}(X^{-1}u)$. In this case, clearly,
$$
X_\alpha(X^{k-m_1(\alpha)}u)M_-^*X_\alpha(X^{k-m_1(\alpha)}u)^{-1}\subseteq M_-^*,
$$
since $M_-^*$ is normal in $E(A[X^{-1}])$.
We also deduce, using Lemma~\ref{lem:E*(A,I)} at the last step, that
\begin{multline*}
\quad X_\alpha(X^{k-m_1(\alpha)}u)X_{-\ha}(IXV_{-\ha})X_\alpha(X^{k-m_1(\alpha)}u)^{-1}=\\
\hbox{$\begin{array}{l}
 =X_\ha(X^{-1}u)X_{-\ha}(IXV_{-\ha})X_\ha(X^{-1}u)^{-1}\\
 =\tau_\ha^{-1}\bigl(X_\ha(u)X_{-\ha}(IV_{-\ha})X_\ha(u)^{-1}\bigr)\\
 \subseteq \tau_\ha^{-1}(E_\ha(A,I))\subseteq M_+^*M_-^*.
\end{array}$}\quad
\end{multline*}
Combining these results, we again have $y\in M_+^*M_-^*$.
\end{proof}

\subsection{Main lemmas.}

\begin{lem}\label{lem:ha-comm}
Assume that $m_1(\ha)=2$. Then for any $u\in V_{\pm\ha}$, one has
$$
X_{\pm\ha}(X^{-1}u)E(A[X],XA[X])\subseteq U_1^\pm(A)EL_1(A)E(A[X],XA[X])X_{\pm\ha}(X^{-1}V_{\pm\ha})
X_{\mp\ha}(XV_{\mp\ha}).
$$
\end{lem}
\begin{proof}
Clearly, it is enough to consider the case of $X_{\ha}(X^{-1}u)$.
By Lemma~\ref{lem:sigma(EX)}, we have
$$
\begin{array}{l}
X_{\ha}(X^{-1}u)E(A[X],XA[X])=\sigma^{-1}\Bigl(X_{\ha}(Xu)\sigma(E(A[X],XA[X]))\Bigr)\\
\qquad\subseteq \sigma^{-1}\Bigl(X_{\ha}(Xu)U_1^-(A)EL_1(A)E(A[X],XA[X])X_{-\ha}(X^{-1}V_{-\ha})\Bigr).
\end{array}
$$
Since $E(A)$ normalizes $E(A[X],XA[X])$, the last expression can be rewritten as
$$
\sigma^{-1}\Bigl(U_1^-(A)EL_1(A)E(A[X],XA[X])X_{-\ha}(X^{-1}V_{-\ha})\Bigr).
$$
We have $\sigma^{-1}\bigl(X_{-\ha}(X^{-1}V_{-\ha})\bigr)=X_{-\ha}(XV_{-\ha})$, and
$$
\sigma^{-1}\bigl(U_1^-(A)EL_1(A)\bigr)=\sigma^{-1}\bigl(U_1^-(A)\bigr)EL_1(A)\subseteq E(A[X],XA[X])EL_1(A).
$$
Applying Lemma~\ref{lem:sigma(EX)} to $\sigma^{-1}\bigl(E(A[X],XA[X])\bigr)$, and using the fact that $E(A)$
normalizes $E(A[X],XA[X])$ again, as well as that $EL_1(A)$ normalizes $U_1^+(A)$, we deduce
$$
\begin{array}{l}
\sigma^{-1}\Bigl(U_1^-(A)EL_1(A)E(A[X],XA[X])X_{-\ha}(X^{-1}V_{-\ha})\Bigr)\subseteq\\
\qquad\subseteq E(A[X],XA[X])EL_1(A)U_1^+(A)E(A[X],XA[X])X_{\ha}(X^{-1}V_\ha)X_{-\ha}(XV_{-\ha})\\
\qquad\subseteq U_1^+(A)EL_1(A)E(A[X],XA[X])X_{\ha}(X^{-1}V_\ha)X_{-\ha}(XV_{-\ha}).
\end{array}
$$
\end{proof}

\begin{lem}\label{lem:sigma-Z}
The product
$$
E(A[X])E(A[X^{-1}])E(A[X])
$$
is invariant under $\sigma^\pm$.
\end{lem}
\begin{proof}
We consider the case of $\sigma$, the case of $\sigma^{-1}$ is symmetric. Set
$$
Z=E(A[X])E(A[X^{-1}])E(A[X]).
$$

Since $E(A)$ normalizes
$E(A[X],XA[X])$ and $E(A[X^{-1}],X^{-1}A[X^{-1}])$, we conclude that
$$
Z=E(A[X],XA[X])E(A)E(A[X^{-1}],X^{-1}A[X^{-1}])E(A[X],XA[X]).
$$
By Gauss decomposition~\cite[Th\'eor\`eme 5.1]{SGA} we have
$$
E(A)=U_1^+(A)U_1^-(A)EL_1(A)U_1^+(A)=U_1^+(A)EL_1(A)U_1^-(A)U_1^+(A).
$$
Then we can rewrite $Z$ as
\begin{equation}\label{eq:Z-reordered}
\begin{array}{l}
Z=U_1^+(A)EL_1(A)E(A[X],XA[X])E(A[X^{-1}],X^{-1}A[X^{-1}])\cdot\\
\qquad\quad\cdot U_1^-(A)E(A[X],XA[X])U_1^+(A).
\end{array}
\end{equation}

Now we consider two cases. First assume that $m_1(\ha)=1$. In this case by Lemma~\ref{lem:sigma(EX)-1}
we have
$$
\sigma\bigl(E(A[X],XA[X])\bigr)\subseteq E(A[X]),\qquad \sigma\bigl(E(A[X^{-1}],X^{-1}A[X^{-1}])\bigr)\subseteq E(A[X^{-1}]).
$$
Here the second inclusion actually follows from the case of $\sigma^{-1}$ in Lemma~\ref{lem:sigma(EX)-1} after interchanging
$X$ and $X^{-1}$. Using these inclusions and the expression~\eqref{eq:Z-reordered} for $Z$, one readily sees
that $\sigma(Z)\subseteq Z$.

From now until the end of the proof, we assume that $m_1(\ha)=2$. In this case we use Lemma~\ref{lem:sigma(EX)}
instead of Lemma~\ref{lem:sigma(EX)-1}. We see from the expression~\eqref{eq:Z-reordered},
that $\sigma(Z)$ is contained in
\begin{equation}\label{eq:s(Z)-1}
 E(A[X])X_{-\ha}(X^{-1}V_{-\ha})\cdot \sigma\Bigl(E(A[X^{-1}],X^{-1}A[X^{-1}])U_1^-(A)\Bigr)\cdot
X_{-\ha}(X^{-1}V_{-\ha}) E(A[X]).
\end{equation}
Using Lemma~\ref{lem:ha-comm}, we compute
$$
\begin{array}{l}
\sigma\Bigl(E(A[X^{-1}],X^{-1}A[X^{-1}])U_1^-(A)\Bigr)
X_{-\ha}(X^{-1}V_{-\ha})\\
\qquad=\sigma\Bigl(E(A[X^{-1}],X^{-1}A[X^{-1}])U_1^-(A)X_{-\ha}(XV_{-\ha})\Bigr)\\
\qquad=\sigma\Bigl(E(A[X^{-1}],X^{-1}A[X^{-1}])X_{-\ha}(XV_{-\ha})U_1^-(A)\Bigr)\\
\qquad\subseteq\sigma\Bigl(X_{\ha}(X^{-1}V_\ha)X_{-\ha}(XV_{-\ha})E(A[X^{-1}],X^{-1}A[X^{-1}])EL_1(A)U_1^-(A)\Bigr)\\
\qquad\subseteq X_{\ha}(XV_\ha)X_{-\ha}(X^{-1}V_{-\ha}) \sigma\Bigl(E(A[X^{-1}],X^{-1}A[X^{-1}])U_1^-(A)\Bigr)EL_1(A).\\
\end{array}
$$
Note that we have used the inclusion from Lemma~\ref{lem:ha-comm} with both sides inverted; we can do that since
$\ha$ is a non-multipliable root, and therefore $\l(X_{\ha}(XV_\ha)\r)^{-1}=X_{\ha}(XV_\ha)$.

Substituting the previous computation into~\eqref{eq:s(Z)-1}, we obtain
\begin{equation}\label{eq:s(Z)-2}
\begin{array}{l}
\sigma(Z)\subseteq E(A[X])\cdot X_{-\ha}(X^{-1}V_{-\ha})X_{\ha}(XV_\ha)X_{-\ha}(X^{-1}V_{-\ha})\cdot\\
\qquad \cdot\sigma\Bigl(E(A[X^{-1}],X^{-1}A[X^{-1}])U_1^-(A)\Bigr)\cdot
E(A[X]).
\end{array}
\end{equation}
By Gauss decomposition in $E_\ha(A)$, we have
$$
\begin{array}{l}
X_{-\ha}(X^{-1}V_{-\ha})X_{\ha}(XV_\ha)X_{-\ha}(X^{-1}V_{-\ha})=\tau_\ha\Bigl(
X_{-\ha}(V_{-\ha})X_{\ha}(V_\ha)X_{-\ha}(V_{-\ha})\Bigr)\\
\qquad\qquad\subseteq
\tau_\ha\Bigl((L_\ha(A)\cap E_\ha(A))X_{\ha}(V_{\ha})X_{-\ha}(V_{-\ha})X_{\ha}(V_{\ha})\Bigr)\\
\qquad\qquad=(L_\ha(A)\cap E_\ha(A))X_{\ha}(XV_{\ha})X_{-\ha}(X^{-1}V_{-\ha})X_{\ha}(XV_{\ha}).
\end{array}
$$
Subtituting this result into~\eqref{eq:s(Z)-2}, we see that to complete the proof of the lemma, it is enough to show
that
$$
X_{\ha}(XV_{\ha})\sigma\Bigl(E(A[X^{-1}],X^{-1}A[X^{-1}])U_1^-(A)\Bigr)\subseteq E(A[X^{-1}])E(A[X]).
$$
We obtain this inclusion using Lemma~\ref{lem:sigma(EX)} again:
$$
\begin{array}{l}
X_{\ha}(XV_{\ha})\sigma\Bigl(E(A[X^{-1}],X^{-1}A[X^{-1}])U_1^-(A)\Bigr)\\
\qquad\qquad =\sigma\Bigl(X_{\ha}(X^{-1}V_{\ha})E(A[X^{-1}],X^{-1}A[X^{-1}])U_1^-(A)\Bigr)\\
\qquad\qquad=\sigma\Bigl(U_1^-(A)E(A[X^{-1}],X^{-1}A[X^{-1}])\Bigr)\subseteq
E(A[X^{-1}])X_{\ha}(XV_\ha).
\end{array}
$$
\end{proof}

\begin{lem}\label{lem:EXX-1}
One has $E(A[X,X^{-1}])=E(A[X])E(A[X^{-1}])E(A[X])$.
\end{lem}
\begin{proof}
Set $Z=E(A[X])E(A[X^{-1}])E(A[X])$.
By Theorem~\ref{th:PS-normality} the elementary subgroup $E(A[X,X^{-1}])$ is generated by the points of the unipotent radicals
of any two opposite parabolic $A$-subgroups of $G$; we take $U_1^+(A[X,X^{-1}])$ and $U_1^-(A[X,X^{-1}])$. In other
words, $E(A[X,X^{-1}])$ is generated by the elements of $X_\alpha(V_\alpha\otimes_A A[X,X^{-1}])$
for all $\alpha\in\Psi$ such that $m_1(\alpha)\neq 0$. Since, clearly, $Z\subseteq E(A[X,X^{-1}])$, it is enough
to show that
$$
X_\alpha(u)Z\subseteq Z\quad\mbox{for any $\alpha\in\Psi$ such that $m_1(\alpha)\neq 0$, $u\in V_\alpha\otimes_A A[X,X^{-1}]$}.
$$
By symmetry of positive and negative roots, it is enough to consider only $\alpha\in\Psi$ with $m_1(\alpha)>0$.
Clearly, for any $u\in V_\alpha\otimes_A A[X,X^{-1}]$ there is $k\ge 0$ such that $X^ku\in V_\alpha\otimes_A A[X]$.
Then, since $Z$ is $\sigma^{\pm 1}$-invariant by Lemma~\ref{lem:sigma-Z}, we have
$$
\begin{array}{rl}
X_\alpha(u)Z&=\sigma^{-k}\bigl(\sigma^k(X_\alpha(u)Z)\bigr)\subseteq \sigma^{-k}\bigl(X_\alpha(X^{km_1(\alpha)}u)Z\bigr)\\
& \subseteq \sigma^{-k}\bigl(E(A[X])Z\bigr)=\sigma^{-k}(Z)\subseteq Z.
\end{array}
$$
\end{proof}

\begin{lem}\label{lem:M*}
One has $M^*=M_+^*M_-^*$.
\end{lem}
\begin{proof}
Let $x\in M^*$. By Lemma~\ref{lem:EXX-1} we have $x=x_1yx_2$, where $x_1,x_2\in E(A[X])$, $y\in E(A[X^{-1}])$.
Since $\rho(x)=1$, we have
$$
\rho(y)=\rho(x_1)^{-1}\rho(x_2)^{-1}\in E(l[X])\cap E(l[X^{-1}])=E(l).
$$
Then $y\in E(A)M_-^*$. By Lemma~\ref{lem:M_+^*} we have $M_-^*E(A[X])\subseteq E(A[X])M_-^*$,
hence $yx_2$ belongs to $E(A[X])M_-^*$, and thus $x=x_1yx_2$ belongs to $E(A[X])M_-^*$. Since $\rho(x)=1$,
we have $x\in M_+^*M_-^*$.
Hence $M^*=M_+^*M_-^*$.
\end{proof}

\section{Proof of the main theorems}

\subsection{$\PP^1$-gluing theorem.}


\begin{proof}[Proof of Theorem~\ref{th:P^1}]
We only need to prove the exactness at the third term. It is equivalent to showing that,
for any $x\in G(A[X],XA[X])$, if there exists an element $y\in G(A[X^{-1}])$ such that $xy^{-1}\in E(A[X,X^{-1}])$, then
$x\in E(A[X])$. We prove this claim.

By Suslin's local-global principle Lemma~\ref{lem:PS-17} we can assume that $A$ is local. Let $I$ be the maximal ideal of $A$,
$l=A/I$, and $\rho:G(A[X,X^{-1}])\to G(l[X,X^{-1}])$ the natural map. By the Margaux---Soul\'e Theorem~\ref{th:MS-red} we have
 $G(l[X])=G(l)E(l[X])$.
Since $x\in G(A[X],XA[X])$, we have $\rho(x)\in E(l[X])$, and hence $x\in E(A[X])G(A[X],I\cdot A[X])$. Therefore, we can
assume $x\in G(A[X],I\cdot A[X])$ from the start.

Then, by the assumption of the theorem, $\rho(y)\in E(l[X,X^{-1}])$ and hence,
using~the Margaux---Soul\'e Theorem~\ref{th:MS-red} again, we obtain
$$
\rho(y)\in G(l[X^{-1}])\cap E(l[X,X^{-1}])=G(l)E(l[X^{-1}])\cap E(l[X,X^{-1}]).
$$
Since $G(l)\cap E(l[X,X^{-1}])=E(l)$, we have $\rho(y)\in E(l)E(l[X^{-1}])=E(l[X^{-1}])$,
and hence
$$
y\in E(A[X^{-1}])G(A[X^{-1}],I\cdot A[X^{-1}]).
$$
Therefore, we can assume that $y\in G(A[X^{-1}],I\cdot A[X^{-1}])$ from the start. Then
$$
xy^{-1}\in G(A[X,X^{-1}],I\cdot A[X,X^{-1}])\cap E(A[X,X^{-1}])=E^*(A[X,X^{-1}],I\cdot A[X,X^{-1}]).
$$
By Corollary~\ref{cor:XX^-1} we have $xy^{-1}=x_+x_-$ for some $x_+\in E(A[X])$, $x_-\in E(A[X^{-1}])$. Therefore,
$x_+^{-1}x=x_-y\in G(A[X])\cap G(A[X^{-1}])=G(A)$. Hence $x\in G(A)E(A[X])$, and
thus $x\in E(A[X])$.
\end{proof}


The following Lemma extends~\cite[Corollary 5.7]{Sus} for $\GL_l$, and~\cite[Prop. 3.3]{Abe} for Chevalley groups.
It will be used in the proof of Theorem~\ref{th:whitehead-field}.

\begin{lem}\label{lem:E_f}
Let $A$, $G$ be as in Theorem~\ref{th:P^1}. Let $f\in A[X]$ be a monic polynomial.
For any $x\in G(A[X],XA[X])$ such that $F_f(x)\in E(A[X]_f)$, one has $x\in E(A[X])$.
\end{lem}
\begin{proof}
The proof literally repeats that of~\cite[Proposition 3.3]{Abe} (or~\cite[Corollary 5.7]{Sus}), using Lemma~\ref{lem:A_fg}
instead of~\cite[Lemma 3.2]{Abe} and Theorem~\ref{th:P^1} instead of~\cite[Theorem 2.16]{Abe}.
\end{proof}

\subsection{$K_1^G$ of the ring of polynomials over a field.}



\begin{proof}[Proof of Theorem~\ref{th:whitehead-field}]
The proof goes by induction on $n$. The case $n=1$ is the Margaux---Soul\'e Theorem~\ref{th:MS-red}.

Assume that $n\ge 2$, and the claim is true for any number of variables less than $n$, for a fixed field $k$.
By Lemma~\ref{lem:sc-red} it is enough to prove the equality
$G(k[X_1,\ldots,X_n])=G(k)E(k[X_1,\ldots,X_n])$
in the case where $G$ is a simply connected semisimple group.
Moreover, exactly as in the proof of Theorem~\ref{th:MS-red},
we reduce to the case where $G$ is a simply connected absolutely almost simple group.

Set $A=k[X_3,\ldots,X_n]$, so that $G(k[X_1,\ldots,X_n])=G(A[X_1,X_2])$.
For any element $x(X_1,X_2)\in G(A[X_1,X_2])$, consider the product
\begin{equation}\label{eq:y(X1,X2)}
y(X_1,X_2)=x(X_1,X_2)x(0,X_2)^{-1}x(0,0)x(X_1,0)^{-1}.
\end{equation}
Clearly, $y(X_1,0)=y(0,X_2)=1$.
Consider the inclusion
$$
G(A[X_1,X_2])\subseteq G\bigl(k(X_1)[X_2,\ldots,X_n]\bigr).
$$
By the induction hypothesis, we have
$$
G\bigl(k(X_1)[X_2,\ldots,X_n]\bigr)=G(k(X_1))E\bigl(k(X_1)[X_2,\ldots,X_n]\bigr).
$$
By~\cite[Th\'eor\`eme 5.8]{Gil} we have $G(k(X_1))=G(k)E(k(X_1))$, hence
$$
G\bigl(k(X_1)[X_2,\ldots,X_n]\bigr)=G(k)E\bigl(k(X_1)[X_2,\ldots,X_n]\bigr).
$$
Then we can write $y(X_1,X_2)=y_0\cdot y_1(X_1,X_2)$, where $y_0\in G(k)$, and
$$
y_1(X_1,X_2)\in E\bigl(k[X_1]_f[X_2,\ldots,X_n]\bigr)=E(A[X_2][X_1]_f),
$$
for a monic polynomial $f(X_1)\in k[X_1]$. Since $y(X_1,0)=1$, we have in fact
$y_0=y_1(X_1,0)^{-1}\in E(A[X_1]_f)$, therefore, $y(X_1,X_2)\in E(A[X_2][X_1]_f)$.
Since $y(X_1,X_2)$ belongs to $G(A[X_1,X_2],X_1A[X_1,X_2])$, by Lemma~\ref{lem:E_f} this implies that
$y(X_1,X_2)\in E(A[X_1,X_2])$. Applying the induction hypothesis to the elements $x(X_1,0)\in G(A[X_1])$,
$x(0,X_2)\in G(A[X_2])$, and $x(0,0)\in G(A)$,
we see that~\eqref{eq:y(X1,X2)} implies $x(X_1,X_2)\in G(k)E(A[X_1,X_2])$.
\end{proof}

\begin{cor}\label{cor:Laurent}
Let $G$ be a simply connected semisimple group scheme over a field $k$, such that every semisimple normal subgroup of
$G$ contains $(\Gm)^2$. For any $m,n\ge 0$, there are natural isomorphisms
$$
K_1^G(k)\cong K_1^G\l(k[Y_1,\ldots,Y_m,X_1,X_1^{-1},\ldots,X_n,X_n^{-1}]\r)
\cong K_1^G\bigl(k(Y_1,\ldots,Y_m,X_1,\ldots,X_n)\bigr).
$$
\end{cor}
\begin{proof}
We prove the first isomorphism by induction on $n$, reducing to the case $n=0$, which is contained in
Theorem~\ref{th:whitehead-field}. After that, the second isomorphism follows from~\cite[Th\'eor\`eme 5.8]{Gil}.

For shortness, we set ${\bf Y}=\{Y_1,\ldots,Y_m\}$.
By the inductive hypothesis and~\cite[Th\'eor\`eme 5.8]{Gil}, we have
$$
K_1^G(k)\cong K_1^G(k(X_1))\cong K_1^G(k(X_1)[{\bf Y},X_2,X_2^{-1},\ldots,X_n,X_n^{-1}]).
$$
Therefore, it is enough to prove that the natural map
$$
K_1^G\l(k[{\bf Y},X_1,X_1^{-1},\ldots,X_n,X_n^{-1}]\r)\to K_1^G\l(k(X_1)[{\bf Y},X_2,X_2^{-1},\ldots,X_n,X_n^{-1}]\r)
$$
 is injective. Set $B=k\l[{\bf Y},X_2,X_2^{-1},\ldots,X_n,X_n^{-1}\r]$. Assume that
$g\in G\l(B[X_1,X_1^{-1}]\r)$ is mapped into $E\l(k(X_1)[{\bf Y},X_2,X_2^{-1},\ldots,X_n,X_n^{-1}]\r)$. Then there exists
a monic polynomial $f\in k[X_1]$ such that
$$
g\in E\l(k[X_1]_{fX_1}[{\bf Y},X_2,X_2^{-1},\ldots,X_n,X_n^{-1}]\r)=
E\l(B[X_1]_{X_1f}\r).
$$
Clearly, we can assume that $f$ is not divided by $X_1$. Then by Lemma~\ref{lem:A_fg} there exist
$g_1\in E(B[X_1]_{X_1})$ and $g_2\in E(B[X_1]_f)$ such that $g=g_1g_2$. The class of $g_1$ in
$K_1^G(B[X_1]_{X_1})=K_1^G(B[X_1,X_1^{-1}])$ is trivial, hence we can assume $g=g_2$.
Since $B[X_1]_{X_1}\cap B[X_1]_{f}=B[X_1]$, we have $g\in G(B[X_1])$.
Then by the inductive hypothesis $g\in E(B[X_1])$.
\end{proof}

\subsection{Homotopy invariance theorem.}\label{sec:geom}

Now we are ready to prove Theorem~\ref{th:whitehead-geometric}.
First we prove it for regular rings essentially of finite type.

\begin{lem}\label{lem:ess-fin}
Let $k,A,G$ be as in Theorem~\ref{th:whitehead-geometric}. Assume in addition that $A$ is of essentially finite type over $k$.
 Then $$
K_1^G(A)\xrightarrow{\cong} K_1^G(A[X]).
$$
\end{lem}
\begin{proof}
The proof is basically the same as the one for $\GL_l$ in~\cite[Theorem 3.1]{Vo}, using the facts
about isotropic groups we have proved above.
Namely, one proceeds by induction on $\dim A$. By Suslin's local-global principle Lemma~\ref{lem:PS-17} we can assume $A$ is
local. If $\dim A=0$, we are in the setting of Theorem~\ref{th:whitehead-field}. Assume $\dim A\ge 1$.
By Lindel's lemma~\cite[Lemma and Proposition 2]{L}
there exists a subring $B$ of $A$ and an element $h\in B$ such that
$B=k[X_1,\ldots,X_n]_p$ for some $n\ge 0$ and a prime ideal $p$ of $k[X_1,\ldots,X_n]$, and $h$ satisfies
$Ah+B=A$, $Ah\cap B=Bh$.

We need to show that any element $a(X)\in G(A[X])$ belongs to $G(A)E(A[X])$.
We can assume from the start that
$a(0)=1$. Since $\dim A_h<\dim A$, the element
$a(X)$ belongs to $G(A_h)E(A_h[X])$. Since $a(0)=1$, we have in fact $a(X)\in E(A_h[X])$.
 Since $A$ is regular, we know that $h$ is not a zero divisor in $A[X]$; hence by Lemma~\ref{lem:Abe-3.7} (i) we have
\begin{equation}\label{eq:a(X)=}
a(X)=y(X)z(X)
\end{equation}
for some $y(X)\in E(A[X])$ and $z(X)\in G(B_h[X])$.
Note that~\eqref{eq:a(X)=} implies that
$$
z(X)\in G(B_h[X])\cap G(A[X])=G(B[X]).
$$
Clearly, we can assume that $z(0)=1$ as well. Since $B$ is a localization of a polynomial ring over $k$,
by Lemma~\ref{lem:Abe-3.6} and Theorem~\ref{th:whitehead-field} we have $z(X)\in E(B[X])$. Therefore,
$a(X)\in E(A[X])$.
\end{proof}

\begin{proof}[Proof of Theorem~\ref{th:whitehead-geometric}]
The embedding $k\to A$ is geometrically regular, since $k$ is perfect~\cite[(28.M), (28.N)]{Mats}.
Then by Popescu's theorem~\cite{Po90,Swan} $A$ is a filtered direct limit of regular $k$-algebras essentially of finite type.
Since the group scheme $G$ and the unipotent radicals of its parabolic subgroups are finitely presented over $k$,
the functors $G(-)$ and $E(-)$ commute with filtered direct limits. Hence the claim of the theorem follows from~Lemma~\ref{lem:ess-fin}.
\end{proof}

\subsection{Injectivity theorem}

\begin{proof}[Proof of Theorem~\ref{th:inj}]
If $k$ is perfect, we can reduce to the case where $A$ is a local ring of a smooth algebraic variety over $k$ in the same way
as in the proof of Theorem~\ref{th:whitehead-geometric}, using Popescu's theorem. To show that $K_1^G(A)\to K_1^G(K)$ is injective for any $A$
of the latter kind, it is enough to check that the functor $K_1^G$ on the category of commutative $k$-algebras satisfies the
conditions of~\cite[Th\'eor\`eme 1.1]{CTO}.
The condition (P1) of~\cite[Th\'eor\`eme 1.1]{CTO} is that $K_1^G$ commutes with filtered direct limits; this is clear.

The condition (P2) requires checking that $K_1^G(L[X_1,\ldots,X_n])\to K_1^G(L(X_1,\ldots,X_n))$ has trivial
kernel for any field $L$ containing $k$ and any $n\ge 1$. By Theorem~\ref{th:whitehead-field} and induction on $n$ it is
enough to check that for any field $L$ as above, if $g\in G(L)$ is mapped into $E(L(X))$, then it belongs to $E(L)$.
There is a polynomial $f\in L[X]$ such that $g\in E(L[X]_f)$. Since $k$ is an infinite field, there is
$a\in L$ such that $f(a)\neq 0$. The evaluation at $X=a$ then shows that $g\in E(L)$.

The condition (P3) is contained in Lemma~\ref{lem:Abe-3.7} (ii).
\end{proof}






\renewcommand{\refname}{References}


\begin{thebibliography}{MMMM}

\bibitem[A]{Abe} E. Abe, {\it Whitehead groups of Chevalley groups over polynomial rings}, Comm. Algebra {\bf 11} (1983),
1271--1307.

\bibitem[BHV]{BHV} A. Bak, R. Hazrat, N. Vavilov, {\it Localizaion---completion strikes again:
Relative $K_1$ is nilpotent by abelian}, J. of Pure and Appl. Algebra {\bf 213} (2009), 1075--1085.

\bibitem[BBR]{BBR} A. Bak, R. Basu, R. A. Rao, {\it Local-global principle for transvection groups},
Proceedings of the AMS {\bf 138} (2010), 1191--1204.




\bibitem[B]{Bass} H.~Bass, {\it K-theory and stable algebra}, Publ. Math. IH\'ES
{\bf 22} (1964), 5--60.











\bibitem[Ba]{Basu} R. Basu, {\it Topics in classical algebraic $K$-theory}, PhD Thesis, 2006.


\bibitem[BT1]{BoTi} A.~Borel, J.~Tits, {\it Groupes r\'eductifs}, Publ. Math. I.H.\'E.S.
{\bf 27} (1965), 55--151.



\bibitem[BT2]{BoTi72}
A.~Borel, J.~Tits, {\it Compl\'{e}ments \`{a} l'article
``Groupes r\'{e}ductifs''}, Publ.~Math.~IH\'ES {\bf 41} (1972) 253--276.

\bibitem[C]{Cohn} P.M. Cohn, {\it On the structure of $GL_2$ of a ring}, Publ.
Math. I.H.\'E.S. {\bf 30} (1966), 365--413.

\bibitem[CTS]{CTS}  J.-L.~Colliot-Th\'el\`ene, J.-J. Sansuc,
{\it Principal homogeneous spaces under flasque tori: applications},
Journal of Algebra {\bf 106} (1987), 148--205.


\bibitem[CTO]{CTO} J.-L. Colliot-Th\'el\`ene, M. Ojanguren, {\it
Espaces Principaux Homog\`enes Localement Triviaux},
Publ.~Math.~IH\'ES {\bf 75}, no.~2 (1992), 97--122.


\bibitem[SGA3]{SGA} M.~Demazure, A.~Grothendieck, {\it Sch\'emas en groupes}, Lecture Notes in
Mathematics, vol. 151--153, Springer-Verlag, Berlin-Heidelberg-New York, 1970.

\bibitem[Ge]{Ge} S. M. Gersten, {\it Higher $K$-theory of Rings}, Lecture Notes in Mathematics, vol. 341,
3--42, Springer-Verlag, Berlin-Heidelberg-New York, 1973.

\bibitem[G]{Gil}  Ph.~Gille, {\it Le probl\`eme de Kneser-Tits}, S\'em. Bourbaki {\bf 983}
(2007), 983-01--983-39.

\bibitem[GMV1]{GMV-Sp} F. Grunewald, J. Mennicke, L. Vaserstein, {\it On symplectic groups over polynomial rings},
Math. Z. {\bf 206} (1991), 35--56.

\bibitem[GMV2]{GMV-SL2} F. Grunewald, J. Mennicke, L. Vaserstein, {\it On the groups $\SL_2(\ZZ[x])$ and
$\SL_2(k[x,y])$}, Israel J. Math. {\bf 86} (1994), 157--193.



\bibitem[J]{J} J.F. Jardine, {\it On the homotopy groups of algebraic groups}, J. Algebra {\bf 81} (1983),
180--201.

\bibitem[K78]{K-stab} V.I. Kopeiko,
{\it Stabilization of symplectic groups over a ring of polynomials} (Russian),
Mat. Sb. (N.S.) {\bf 106 (148)} (1978), no. 1, 94--107.

\bibitem[K95a]{K-reg} V.I. Kopeiko, {\it
On the structure of the symplectic group of polynomial rings over regular rings} (Russian),
Fundam. Prikl. Mat. {\bf 1} (1995), no. 2, 545--548.

\bibitem[K95b]{K-SL-Laur} V.I. Kopeiko, {\it
On the structure of the special linear group over Laurent polynomial rings} (Russian),
 Fundam. Prikl. Mat. {\bf 1} (1995), no. 4, 1111--1114.


\bibitem[K96]{K-lett} V.I. Kopeiko, {\it
Letter to the editors: "On the structure of the symplectic group of polynomial rings over regular rings"{} and
"On the structure of the special linear group over Laurent polynomial rings"{}} (Russian),
Fundam. Prikl. Mat. {\bf 2} (1996), no. 3, 953.

\bibitem[K99]{K-Sp-Laur} V.I. Kopeiko, {\it
Symplectic groups over rings of Laurent polynomials, and patching diagrams} (Russian),
Fundam. Prikl. Mat. {\bf 5} (1999), no. 3, 943--945.

\bibitem[KrMC]{KrMc} S. Krsti\'c, J. McCool, {\it Free quotients of $\SL_2(R[X])$}, Proc. Amer. Math. Soc. {\bf 125}
(1997), 1585--1588.


\bibitem[L]{L} H. Lindel, {\it On the Bass---Quillen conjecture concerning projective
modules over polynomial rings}, Invent. Math. 65 (1981), 319--323.

\bibitem[LSt]{LS} A. Luzgarev, A. Stavrova, {\it Elementary subgroup of an isotropic reductive group is perfect},
St. Petersburg Math. J. {\bf 23} (2012), 881--890.

\bibitem[M]{M} B. Margaux, {\it The structure of the group $G(k[t])$: Variations on a theme of Soul\'e},
Algebra and Number Theory {\bf 3} (2009), 393--409.

\bibitem[Ma]{Mats} H. Matsumura, {\it Commutative algebra}, second ed.,
Math. Lect. Note Series {\bf 56}, Benjamin/Cummings Publishing Co., Inc., Reading, Massachusetts, 1980.

\bibitem[Mo]{Mo} F. Morel, {\it $\Aff^1$-Algebraic topology over a field}, Lecture Notes in Mathematics, vol. 2052, 2012.



\bibitem[PaStV]{PaSV} I. Panin, A. Stavrova, N. Vavilov,
{\it On Grothendieck---Serre's conjecture concerning
principal $G$-bundles over reductive group schemes: I}, preprint,
\href{http://www.arxiv.org/abs/0905.1418}{http://www.arxiv.org/abs/0905.1418}.


\bibitem[PSt1]{PS} V. Petrov, A. Stavrova, {\it Elementary subgroups of isotropic reductive groups},
St. Petersburg Math. J. {\bf 20} (2009), 625--644.

\bibitem[PSt2]{PS-tind} V. Petrov, A. Stavrova, {\it Tits indices over semilocal rings}, Transf. Groups {\bf 16} (2011), 193--217.

\bibitem[Po]{Po90} D. Popescu, {\it Letter to the Editor: General N\'eron desingularization and approximation},
Nagoya Math. J. {\bf 118} (1990), 45--53.

\bibitem[Q]{Q} D.~Quillen, {\it Projective modules over polynomial rings}, Invent. Math. {\bf 36} (1976),
167--171.




\bibitem[Sou]{Sou} C. Soul\'e, {\it Chevalley groups over polynomial rings},
Homological group theory (Proc. Sympos., Durham, 1977), 359--367, London Math. Soc. Lecture Note Ser.
{\bf 36} (1979), Cambridge Univ. Press.

\bibitem[St]{S-thes} A.~Stavrova, Stroenije isotropnyh reduktivnyh grupp, PhD thesis, St. Petersburg
State University, 2009.




\bibitem[Ste]{Step} A. Stepanov, private communication.

\bibitem[Su]{Sus} A.A.~Suslin,
{\it On the structure of the special linear group over polynomial rings}, Math. USSR Izv. {\bf 11}
(1977), 221--238.

\bibitem[SuK]{Sus-K-O1} A.A. Suslin, V.I. Kopeiko, {\it
Quadratic modules and the orthogonal group over polynomial rings},
J. of Soviet Math. {\bf 20} (1982), 2665--2691.

\bibitem[Sw]{Swan} R. G. Swan, {\it  N\'eron-Popescu desingularization}, in Algebra and Geometry (Taipei, 1995),
Lect. Alg. Geom. {\bf 2} (1998), 135--198. Int. Press, Cambridge, MA.

\bibitem[T1]{Tits64} J.~Tits, {\it Algebraic and abstract simple groups}, Ann. of Math. {\bf 80}
(1964), 313--329.


\bibitem[T2]{Tits66} J.~Tits, {\it Classification of algebraic semisimple groups}, Algebraic groups and
discontinuous subgroups, Proc. Sympos. Pure Math. {\bf 9}, Amer. Math. Soc., Providence RI, 1966, 33--62.

\bibitem[vdK]{vdK} W.~van der Kallen, {\it A module structure on certain orbit sets of unimodular rows},
J. Pure Appl. Algebra {\bf 57} (1989), 281--316.

\bibitem[VW]{VW} K. V\"olkel, M. Wendt, {\it On $\Aff^1$-fundamental groups of isotropic reductive groups}, 2012,
\href{http://arxiv.org/abs/1207.2364}{http://arxiv.org/abs/1207.2364}.


\bibitem[V]{Vo} T. Vorst, The general linear group of polynomial rings over regular rings,
Comm. Algebra {\bf 9} (1981), 499--509.

\bibitem[W]{W10} M. Wendt, $\Aff^1$-homotopy of Chevalley groups, J. K-Theory {\bf 5} (2010), 245--287.











\end{thebibliography}
\end{document}